\documentclass[a4paper,12pt]{amsart}

\usepackage[utf8]{inputenc}
\usepackage[T1]{fontenc}
\usepackage{lmodern}
\usepackage[english]{ babel}
\usepackage{microtype}

\usepackage{amsmath,amssymb,amsfonts,amsthm}
\usepackage{mathtools,accents,comment}
\usepackage{mathrsfs}
\usepackage{aliascnt}
\usepackage{braket}
\usepackage{bm}
\usepackage{esint}

\usepackage[a4paper,margin=3cm]{geometry}
\usepackage[citecolor=blue,colorlinks]{hyperref}

\usepackage{enumerate}
\usepackage{xcolor}

\makeatletter
\g@addto@macro\@floatboxreset\centering
\makeatother

\makeatletter
\def\newaliasedtheorem#1[#2]#3{
  \newaliascnt{#1@alt}{#2}
  \newtheorem{#1}[#1@alt]{#3}
  \expandafter\newcommand\csname #1@altname\endcsname{#3}
}
\makeatother

\numberwithin{equation}{section}

\newtheoremstyle{slanted}{\topsep}{\topsep}{\slshape}{}{\bfseries}{.}{.5em}{}

\theoremstyle{plain}
\newtheorem{theorem}{Theorem}[section]
\newaliasedtheorem{proposition}[theorem]{Proposition}
\newaliasedtheorem{lemma}[theorem]{Lemma}
\newaliasedtheorem{corollary}[theorem]{Corollary}
\newaliasedtheorem{counterexample}[theorem]{Counterexample}

\theoremstyle{definition}
\newaliasedtheorem{definition}[theorem]{Definition}
\newaliasedtheorem{question}[theorem]{Question}
\newaliasedtheorem{claim}[theorem]{Claim}
\newaliasedtheorem{example}[theorem]{Example}
\newaliasedtheorem{conjecture}[theorem]{Conjecture}

\theoremstyle{remark}
\newaliasedtheorem{remark}[theorem]{Remark}

\newcommand{\setN}{\mathbb{N}}

\newcommand{\setR}{\mathbb{R}}

\newcommand{\setZ}{\mathbb{Z}}

\newcommand{\cE}{\mathcal{E}}
\newcommand{\cV}{\mathcal{V}}

\newcommand{\eps}{\varepsilon}

\let\phi\varphi

\newcommand{\di}{\mathop{}\!\mathrm{d}}

\newcommand{\loc}{{\rm loc}}

\DeclareMathOperator{\supp}{supp}

\newcommand{\Ch}{{\sf Ch}}

\DeclareMathOperator{\Lip}{Lip}

\DeclareMathOperator{\UG}{\mathrm{UG}}

\newcommand{\dist}{\mathsf{d}}

\newcommand{\meas}{\mathfrak{m}}

\DeclareMathOperator{\CD}{CD}

\DeclareMathOperator{\vol}{\mathrm{vol}}

\DeclareMathOperator{\PI}{PI}

\newfont{\tmpf}{cmsy10 scaled 2500}

\title[Adimensional weighted Sobolev inequalities]{Adimensional weighted Sobolev inequalities in PI spaces}
\author{David Tewodrose}
\address{D. Tewodrose, Université Libre de Bruxelles, Service d’Analyse, CP 218, Boulevard du Triomphe, B-1050 Bruxelles, Belgique.}
\email{David.Tewodrose@ulb.ac.be}

\begin{document}

\maketitle

\begin{abstract}
We provide a family of global weighted Sobolev inequalities and Hardy inequalities on $\PI$ spaces with possibly non-maximal volume growth. Our results apply notably to non-trivial Ahlfors regular spaces like Laakso spaces and Kleiner-Schioppa spaces.
\end{abstract}

\tableofcontents

\section{Introduction}

On a smooth complete non-compact Riemannian manifold $(M,g)$ of dimension $n>2$ equipped with the canonical Riemannian volume measure $\vol_g$, the classical Sobolev inequality writes as
\begin{equation}\label{eq:Sobolev}
\left( \int_M f^{2^\star} \di \vol_g \right)^{2/2^\star} \le C_S \int_M |\nabla f|^{2} \di \vol_g
\end{equation}
for any $f \in H^{1,2}(M)$, where $2^\star:=2n/(n-2)$ is such that $1/2 - 1/2^\star = 1/n$, and $C_S>0$ does not depend on $f$. When the manifold has non-negative Ricci curvature, a necessary condition for this inequality to hold is the existence of a constant $C>0$ such that 
\begin{equation}\label{eq:volgro}
\vol_g(B_r(o)) \ge C r^n
\end{equation}
for some $o \in X$ and any $r\ge 1$, i.e.~the manifold must have Euclidean volume growth. In case $n$ in \eqref{eq:volgro} is replaced by some $\eta \in [1,n)$, the Sobolev inequality does not hold (see e.g.~\cite[Prop.~2.21]{MR2481740}), but under the stronger assumption
\begin{equation}\label{eq:rev}
\frac{\vol_g(B_R(o))}{\vol_g(B_r(o))} \ge C \left( \frac{R}{r} \right)^\eta \qquad \forall \, 0\le r\le R<+\infty
\end{equation}
a \textit{weighted} Sobolev inequality does, namely
\begin{equation}\label{eq:weighted}
\left( \int_M f^{2^\star} \di \mu \right)^{2/2^\star} \le C_{W} \int_M |\nabla f|^{2} \di \vol_g
\end{equation}
for any $f \in H^{1,2}(M)$, where $\mu$ is the measure absolutely continuous with respect to $\vol_g$ with density $w:=\vol(B_{r_o(\cdot)}(o))^{2^\star/n}r_o(\cdot)^{-2^\star}$, the function $r_o(\cdot)$ is the Riemannian distance to $o$, and $C_W$ does not depend on $f$.  Assumption \eqref{eq:rev} is usually called a reverse doubling condition. 

Inequality \eqref{eq:weighted} was established by V.~Minerbe in \cite{MR2481740} who built his proof upon a patching procedure going back at least to the work of A.~Grigor'yan and L.~Saloff-Coste \cite{MR2149405}. Later on, H.-J.~Hein extended \eqref{eq:weighted} to the class of Riemannian manifolds with quadratically decaying Ricci curvature \cite{MR2801635}, after what the author provided another extension of \eqref{eq:weighted} to possibly non-smooth metric measure spaces satisfying the Lott-Sturm-Villani $\CD(0,N)$ curvature-dimension condition \cite{Tew}. 

The goal of this note is to establish a family of weighted Sobolev inequality extending \eqref{eq:weighted} to the general context of complete $\PI_{p}$ metric measure spaces $(X,\dist,\meas)$ satisfying a reverse doubling condition.  Here and throughout the note ``$\PI_{p}$'' qualifies spaces supporting a global doubling condition and a weak $(p,p)$ Poincaré inequality for some $p \ge 1$: we refer to Section 2 for the precise definitions of these two conditions.  Moreover, all the metric measure spaces considered in this note are always tacitly assumed separable.

Before stating our main results, let us make three preliminary remarks. First, \eqref{eq:weighted} is an $L^2$ weighted Sobolev inequality, and we prove $L^p$ versions of this inequality for any $1 \le p < Q$, where $Q$ is the doubling dimension of the space whose definition we recall in Section 2.  Secondly,  the weight in \eqref{eq:weighted} involves the dimension $n$ of the manifold; in order to remove this $n$ we introduce two parameters $s \le t$ such that
\begin{equation}\label{eq:s,t}
0 \le \frac{1}{s}-\frac{1}{t} \le \frac{1}{Q}
\end{equation}
and formulate our inequalities in terms of these parameters: this leads to what we call \textit{adimensional} inequalities. On a Riemannian manifold whose doubling dimension equals the Hausdorff dimension (i.e.~$Q=n$),  the second inequality in \eqref{eq:s,t} is an equality when $s=2$ and $t=2^\star$.  Thirdly, an important feature of our results is that they do not involve any curvature condition: this is a major improvement compared to \cite{MR2481740,MR2801635,Tew}.


Throughout the paper we write $C=C(a_1,\ldots,a_k)$ to express that a constant $C$ depends on the parameters $a_1,\ldots,a_k$ only.  Here is the main theorem of this note.

\begin{theorem}\label{eq:main}
Let $(X,\dist,\meas)$ be a complete $\PI_{p}$ space with doubling dimension $Q>1$, Poincaré exponent $p \ge 1$ and Poincaré constants $C_P>0$ and $\lambda\ge 1$. Assume that $p<Q$ and set $p^{\star}:=\frac{pQ}{Q - p}\,$. Assume that there exists $o \in X$ and $\eta \in (p,Q]$ such that:
\begin{equation}\label{eq:reverse}
\frac{\meas(B_R(o))}{\meas(B_r(o))} \ge C_o \left( \frac{R}{r} \right)^\eta \qquad \forall \,  0<r\le R<+\infty.
\end{equation}
Then for any $s \in [p,\eta)$ and $t \in [s,p^{\star}]$ there exists $C=C(Q,p,C_P,\lambda,\eta,C_o,s,t)>0$ such that
\begin{equation}\label{eq:weightedSobolev}
\left( \int_{X} |f|^{t} \di \mu_{s,t} \right)^{1/t} \le C \Ch_s(f)^{1/s}
\end{equation}
for any $f \in H^{1,s}(X,\dist,\meas)$, where $\di \mu_{s,t}(\cdot):=\meas(B_{\dist(o,\cdot)}(o))^{t/s-1}\dist(o,\cdot)^{-t}\di\meas(\cdot)$.
\end{theorem}

We refer to \eqref{eq:Cheeger} and \eqref{eq:Sob} in Section 2 for the definition of the $s$-Cheeger energy $\Ch_s$ and its associated class of Sobolev functions $H^{1,s}(X,\dist,\meas)$.

Note that when $(X,\dist,\meas)=(M^n,g)$ and $(s,t)=(2,2^\star)$, then \eqref{eq:weightedSobolev} reduces to \eqref{eq:weighted}.

It is well-known that the doubling condition implies the reverse doubling one \eqref{eq:reverse} for some $\eta=\eta(C_D) \le 1$, see e.g.~\cite[Prop.~3.3]{MR3276002}, but in \eqref{eq:reverse} we assume $\eta > p > 1$ which is not a consequence of the doubling condition.

We call \textit{adimensional} the inequalities \eqref{eq:weightedSobolev} because they do not involve any notion of dimension; the doubling dimension $Q$ does provide an upper bound on the Hausdorff dimension of $(X,\dist)$ and the exponent $\eta$ bounds the Hausdorff dimension of any tangent cone at infinity of $(X,\dist)$ from below, but these weak notions of dimension intervene only in the critical values the dimension free parameters $t$ and $s$ can assume.


When $\eta=Q$, Theorem \ref{eq:main} applies to the case of Ahlfors $Q$-regular spaces, namely those metric measure spaces $(X,\dist,\meas)$ for which there exists a constant $C_A\ge 1$ such that $C_A^{-1} r^Q \le \meas(B_r(x)) \le C_A r^Q$ for any $x \in X$ and $r>0$.

\begin{corollary}\label{cor:Ahl}
Let $(X,\dist,\meas)$ be a complete Ahlfors $Q$-regular space satisfying a weak $(p,p)$ Poincaré inequality with $1 \le p<Q$. Then for any $o\in X$, any $s \in [p,Q)$ and any $t \in [s,p^\star]$,  there exists $C=C(Q,p,C_P,s,t)>0$ such that
\begin{equation}\label{eq:weightedSobolev2}
\left( \int_{X} |f|^{t} r_o^{Q(t/s-1)-1} \di \meas \right)^{1/t} \le  C_A^{1/s-1/t} C \Ch_s(f)^{1/s}
\end{equation}
for any $f \in H^{1,s}(X,\dist,\meas)$, where $r_o(\cdot):= \dist(o,\cdot)$.
\end{corollary}

Our proof of Theorem \ref{eq:main} follows readily the lines of \cite{MR2481740, Tew}, but we provide details for the reader's convenience. We patch local Sobolev inequalities -- derived from the PI condition like in \cite{HajlaszKoskela} -- by means of a discrete Poincaré inequality. This latter inequality holds on a graph whose structure reflects the geometry of the space.

Patching local Poincaré inequalities instead of local Sobolev inequalities yields the following Hardy inequalities.

\begin{theorem}\label{th:hardy}
Let $(X,\dist,\meas)$, $Q$, $p$, $o$ and $\eta$ be as in Theorem \ref{eq:main}. Set $r_o(\cdot):=\dist(o,\cdot)$. Then for any $s \in [p,\eta)$, there exists $C=C(Q,C_o,\eta,s)>0$ such that for any $f \in H^{1,s}(X,\dist,\meas)$,
\begin{equation}\label{eq:Hardy}
\int_{X} |f|^{s} r_o^{-s} \di \meas \le C \Ch_s(f).
\end{equation}
\end{theorem}

Let us finally single out two important classes of Ahlfors regular spaces falling in the framework of our results. The first one was introduced in \cite{MR1748917} by T.~Laakso who constructed, for any $Q>1$, an Ahlfors $Q$-regular metric measure space $(X_Q,\dist_Q,\meas_Q)$ supporting a weak $(1,1)$ Poincaré inequality. Though T.~Laakso detailed only the construction of compact spaces, he suggested a natural modification to get non-compact spaces with the same properties. The second one was introduced by B.~Kleiner and A.~Schioppa in \cite{MR3641484}. For any integer $n \ge 1$, they built a non-compact Ahlfors $Q$-regular space $(X_n,\dist_n,\meas_n)$, with $Q=1+(n-1)/\alpha$ and $\alpha$ suitably chosen in $(0,1)$, satisfying a weak $(1,1)$ Poincaré inequality. Moreover, each $(X_n,\dist_n,\meas_n)$ has topological dimension $n$, Hausdorff dimension $Q$, and analytic dimension $1$ (meaning that Cheeger's cotangent bundle \cite{MR1708448} has dimension $1$); in this regard, our adimensional inequalities are likely to help investigating the analytic properties of these spaces.\\

\textbf{Acknowledgments.}
I thank T.~Coulhon for the initial impetus he gave to this work, L.~Mari for precious comments on two  successive versions of this article, and the two anonymous referees for their kind and helpful suggesions.

\section{Analysis tools on metric measure spaces}

We call metric measure space any triple $(X,\dist,\meas)$ where $(X,\dist)$ is a separable metric space and $\meas$ a non-negative Borel regular measure on $(X,\dist)$ which is finite and non-zero on balls with finite and non-zero radius.  Whenever $A$ is a subset of a metric space, we write $\overline{A}$ for its closure. If $\meas(A)\neq 0$ and $f$ is integrable on $A$, we write $f_A$ or $\fint_A f \di \meas$ for the mean value $\meas(A)^{-1} \int_A f \di \meas$, or $f_{A,\meas}$ to emphasize the measure with respect to which the mean value is taken. We use the classical notation $B_r(x)$ for the open ball centered at $x$ with radius $r$, $S_r(x)$ for the sphere centered at $x$ with radius $r$, and $A(x,r,R)$ for the annulus $B_R(x)\backslash \overline{B_r(x)}$, where $0<r \le R$.  Unless otherwise stated,  whenever $\Gamma$ is a set we write $\#\Gamma$ for its cardinality.

We use classical notations for function spaces, like $C(X)$ for the space of continuous functions, $C_c(X)$ for the one of compactly supported continuous functions, $\Lip(X)$ for Lipschitz functions, $\Lip_{bs}(X)$ for Lipschitz functions with bounded support, and if $A$ is a $\meas$-measurable subset of $X$ we use $L^p(A,\meas)$ (resp.~$L^p_{loc}(A,\meas)$) for (resp.~locally) $p$-integrable functions on $A$, with $1\le p<+\infty$, $L^\infty(A,\meas)$ for $\meas$-essentially bounded functions on $A$ and $L^0(A,\meas)$ (resp.~$L^0_+(A,\meas)$) for (resp.~non-negative) $\meas$-measurable functions on $A$.

We recall that for any $p\ge 1$ the function $t \mapsto |t|^p$ is convex, so that for any $a,b \in \setR$,
 \begin{equation}\label{eq:2p}
 |a+b|^p\le2^{p-1}(|a|^p+|b|^p),
 \end{equation}
 and for any $x_1,\ldots,x_k \ge 0$,
  \begin{equation}\label{eq:p}
\sum_i x_i^p \le \left(\sum_i x_i\right)^p.
 \end{equation}

\hfill

\textbf{Upper gradients and $p$-Cheeger energies}

Let $(X,\dist,\meas)$ be a metric measure space. For any $f:X \to [-\infty,+\infty]$, a Borel function $g : X \to [0,+\infty]$ is called an upper gradient of $f$ if for any rectifiable curve $\gamma:[0,L] \rightarrow X$ parametrized by arc-length,
\[
|f(\gamma(L))-f(\gamma(0))| \le \int_0^L g(\gamma(s)) \di s.
\]
For any $p \in [1,+\infty)$, we write $\UG^p(f)$ for the set of $p$-integrable upper gradients of $f$. When $f \in L^p(X,\meas)$, its Cheeger $p$-energy \cite{MR1708448} is set as
\begin{equation}\label{eq:Cheeger}
\Ch_p(f):= \inf \left\{\liminf\limits_{i \to \infty} \|g_{i}\|_{L^{p}}^p \right\} \, \, \, \in [0,+\infty],
\end{equation}
where the infimum is taken over all the sequences $\{f_i\}_i \subset L^p(X,\meas)$ and $\{g_{i}\}_i \subset L^0_+(X,\meas)$ such that $g_{i}$ is an upper gradient of $f_{i}$ and $\| f_i - f \|_{L^p} \to 0$. Then the Sobolev space $H^{1,p}(X,\dist,\meas)$ is set as the finiteness domain of $\Ch_p$:
\begin{equation}\label{eq:Sob}
H^{1,p}(X,\dist,\meas) := \{ f \in L^p(X,\meas) \, : \, \Ch_p(f) < +\infty\}.
\end{equation}
It is a Banach space when equipped with the norm
$$
\|\cdot\|_{H^{1,p}} := (\|\cdot\|_{L^p}^p + \Ch_p(\cdot))^{1/p}.
$$
Moreover, for any function $f \in H^{1,p}(X,\dist,\meas)$, there exists a function $|\nabla f|_{*,p} \in L^p(X,\meas)$ called minimal $p$-relaxed upper
gradient of $f$ such that $|\nabla f|_{*,p} \in \UG(f)$, $$\Ch_p(f)=\int_X |\nabla f|_{*,p}^p \di \meas,$$
and $|\nabla f|_{*,p} \le g$ $\meas$-a.e.~for any $g \in \UG(f)$.

\begin{remark}
In the literature, $\Ch_p$ is sometimes defined with a coefficient $1/p$ in front, but this coefficient does not play any role in our discussion, so we skip it.
\end{remark}

\hfill

\textbf{PI spaces} 

A metric measure space $(X,\dist,\meas)$ is called a PI space if the two following properties are in force:
\begin{enumerate}
\item[(i)] (global doubling condition) there exists $C_D\ge 1$ such that for any $x \in X$ and $r>0$,
\begin{equation}\label{eq:doubling}
\meas(B_{2r}(x)) \le C_D \meas(B_r(x));
\end{equation}
\item[(ii)] (weak $(q,p)$ Poincaré inequality) there exists $p,q \ge 1$, $\lambda\ge 1$ and $C_P>0$ such that for all $f \in L^1_{loc}(X,\meas)$, all $g \in \UG^p(f)$, and all ball $B_r$ with radius $r>0$,
$$
\left( \fint_{B_r} |f - f_{B_r}|^q \di \meas \right)^{q} \le C_P r \left( \fint_{B_{\lambda r}} g^p \di \meas \right)^{1/p}\, .
$$
\end{enumerate}
We may use the notation $\PI_{q,p}$ to underline the values of the exponent involved in the weak Poincaré inequality, and write $\PI_p$ in case $q=p$. 

It is worth pointing out that the weak $(q,p)$ Poincaré inequality implies the $(q,p')$ and the $(q',p)$ ones with same constant $C_P$ for any $p' \ge p$ and $q' \le q$, as Hölder's inequality implies
\[
\left( \fint_{B_{\lambda r}} g^p \di \meas \right)^{1/p} \le \left( \fint_{B_{\lambda r}} g^{p'} \di \meas \right)^{1/p'}
\]
and
\[
\left( \fint_{B_r}|f - f_{B_r}|^{q'} \di \meas \right)^{1/q'} \le\left( \fint_{B_r}|f - f_{B_r}|^{q} \di \meas \right)^{1/q}.
\]

If only (i) is true, we say that $(X,\dist,\meas)$ is a doubling space, or we say that $\meas$ is a doubling measure; in this case, a simple argument shows that
$$
\frac{\meas(B_R(x))}{\meas(B_r(x))} \le C_D \left( \frac{R}{r} \right)^{\log_2 C_D}
$$
for any $0<r \le R<+\infty$: see \cite[Prop.~3.2]{MR3276002}, for instance.  This form of the doubling condition may clarify why \eqref{eq:reverse} is called a reverse doubling condition. The constant $Q:=\log_2 C_D$ where $C_D$ is the lowest constant such that \eqref{eq:doubling} holds is called the doubling dimension of $(X,\dist,\meas)$.  Moreover, as a consequence of the strong annular decay property \cite{Buckley}, any doubling metric measure space satisfies 
\begin{align}\label{eq:sphèresnégli}
\meas(S_r(x))=0
\end{align}
for any $x \in X$ and $r>0$.

 

Following \cite{MR3363168}, we say that a metric space $(X,\dist)$ is geodesic, or equivalently that $\dist$ is a geodesic distance on $X$, if every pair of points in $X$ can be joined by a rectifiable curve whose length is equal to the distance between
the points. Such a curve is called a geodesic. Any complete PI space carries a geodesic distance that is bi-Lipschitz equivalent to the original metric: see \cite[Cor.~8.3.16]{MR3363168}. Since the PI property is preserved by bi-Lipschitz equivalence of metrics \cite[Lem.~8.3.18]{MR3363168}, with no loss of generality we can (and will from now on) assume that any complete PI space is geodesic.


\hfill

\textbf{Good coverings and patching theorem}

We call doubly measured metric space any quadruple $(X,\dist,\meas,\mu)$ where $(X,\dist,\meas)$ and $(X,\dist,\mu)$ are metric measure spaces, and we write $\Ch_p$ for the $p$-Cheeger energy of $(X,\dist,\meas)$. For simplicity, we assume that $\supp(\meas)=\supp(\mu)=X$, but this assumption can be removed with the help of a natural adaptation of the definitions below. We provide the definition of a good covering \cite{MR2149405, MR2481740,Tew}.

\begin{definition}\label{def:goodcovering}
A good covering of a doubly measured metric space $(X,\dist,\meas,\mu)$ is a countable family of triples $\{(U_{i},U_{i}^{*},U_{i}^{\#})\}_{i \in I}$ of Borel sets with finite $\meas$ and $\mu$ measure such that
\begin{enumerate}
\item $U_i\subset U_i^* \subset U_i^{\#}$ for any $i$,  
\item $\meas(X \backslash \bigcup_{i} U_{i})=\mu(X \backslash \bigcup_{i} U_{i})=0$,
\end{enumerate}
and for some $Q_1,Q_2>0$,
\begin{enumerate}
\item[(3)] $\sup_i  \# \{ j \in I :  \overline{U_{j}^{\#}} \cap \overline{U_{i}^{\#}} \neq \emptyset \} \le Q_{1}$,
\item[(4)] whenever $i \sim j $ there exists $k(i,j) \in I$ such that $U_{i} \cup U_{j} \subset U_{k(i,j)}^{*}$ and $\meas(U_{k(i,j)}^{*}) \le Q_{2} \min (\meas(U_{i}),\meas(U_{j}))$, $\mu(U_{k(i,j)}^{*}) \le Q_{2} \min (\mu(U_{i}),\mu(U_{j}))$.
\end{enumerate}
\end{definition}

Here and in the sequel $i \sim j$ means $\overline{U_{j}} \cap \overline{U_{i}} \neq \emptyset.$ It is worth pointing out that if $\{(U_{i},U_{i}^{*},U_{i}^{\#})\}_{i \in I}$ is a good covering of $(X,\dist,\meas,\mu)$, then 
\begin{equation}\label{eq:Q1}
\sum_{i \in \cV} 1_{U_i^*}  \le Q_1
\end{equation}
because $U_i^*\cap U_j^* \neq \emptyset$ implies $\overline{U_{j}^{\#}} \cap \overline{U_{i}^{\#}} \neq \emptyset$, and
\begin{equation}\label{eq:Q12}
\sum_{i \in \cV} \sum_{i \sim j} 1_{U^*_{k(i,j)}}\le  Q_1^3
\end{equation}
because $$\sum_{i \sim j} 1_{U^*_{k(i,j)}} \le Q_1   1_{\cup_{j \sim i}U^*_{k(i,j)}} \le Q_1  1_{\cup_{j \sim i}U^\#_{k(i,j)}}$$ for any $i \in \cV$ and
$$
\sum_{i \in \cV}  1_{\cup_{j \sim i}U^\#_{k(i,j)}} \le Q_1^2.
$$
\begin{definition}\label{def:weightedgraph}
The weighted graph $(\cV,\cE,\mu)$ canonically attached to a good covering $\{(U_{i},U_{i}^{*},U_{i}^{\#})\}_{i \in I}$ is obtained by setting  $\cV := I$, $\cE := \{ (i,j) \in I^2 : i \sim j\}$, and $\mu : \cV \sqcup \cE  \rightarrow [0,+\infty)$ where $\mu(i) := \mu(U_{i})$ for every $i \in I$ and  $\mu(i,j) := \min (\mu(U_{i}),\mu(U_{j}))$ for every $(i,j) \in \cE$. 
\end{definition}

\begin{remark}\label{rk:measure}
If $(\cV,\cE,\mu)$ is a weighted graph as in the previous definition, then $\mu$ provides a measure on $\cV$, still denoted by $\mu$, defined by $\mu(\Omega):=\sum_{i \in \Omega} \mu(i)$ for any $\Omega \subset \cV$, and another measure on $\cE$, denoted by $\mu$ too, defined by $\mu(\Gamma):=\sum_{(i,j)\in \Gamma} \mu(i,j)$ for any $\Gamma \in \cE$. 
\end{remark}


For $1 \le s \le t < +\infty$, we say that a good covering $\{(U_{i},U_{i}^{*},U_{i}^{\#})\}_{i \in I}$ satisfies local $(t,s)$  continuous Sobolev-Neumann inequalities if there exists a constant $C_1>0$ such that for any $i \in I$,
\begin{equation}\label{continuous Sobolev-Neumann 1}
\left( \int_{U_{i}} |f-f_{U_{i},\mu}|^t \di \mu \right)^{1/t} \le C_1 \left( \int_{U_i^*} g^s \di \meas\right)^{1/s}
\end{equation}
for any $f \in L^1_{loc}(U_i,\mu)$ and $g \in UG(f)\cap L^s(U_i^*,\meas)$, and 
\begin{equation}\label{continuous Sobolev-Neumann 2}
\left( \int_{U_{i}^{*}} |f-f_{U_{i}^{*},\mu}|^t \di \mu \right)^{1/t} \le C_1 \left( \int_{U_i^\#} g^s \di \meas\right)^{1/s}
\end{equation}
for any $f \in L^1_{loc}(U_i^*,\mu)$ and $g \in UG(f)\cap L^s(U_i^\#,\meas)$; we say that $\{(U_{i},U_{i}^{*},U_{i}^{\#})\}_{i \in I}$ satisfies a $t$ Poincaré inequality if the associated weighted graph $(\cV,\cE,\mu)$ admits a constant $C_2>0$ such that for any $f:\cV \to \setR$ with finite support,
\begin{equation}\label{discrete Poincaré}
\left( \sum_{i \in \cV} |f(i)|^{t} \mu(i) \right)^{1/t} \le C_2 \left( \sum_{(i,j) \in \cE} |f(i) - f(j)|^{t} \mu(i,j)  \right)^{1/t}.
\end{equation}

\begin{remark}\label{rk:classicalnotation}
For any $f : \cV \to \setR$ one can define $|\nabla f| : \cE \to \setR$ by setting $|\nabla f| (i,j) := |f(i)-f(j)|$ for any $(i,j)\in \cE$. In this way, \eqref{discrete Poincaré} takes the more classical form 
$\|f\|_{L^{t}(\cV,\mu)} \le C_2 \||\nabla f|\|_{L^{t}(\cE,\mu)}$.
\end{remark}

Then the following \textit{patching} theorem holds.

\begin{theorem}\label{th:patching}
Let $(X,\dist,\meas,\mu)$ be a doubly measured metric space admitting a good covering satisfying the local $(t,s)$ continuous Sobolev-Neumann inequalities and the discrete $t$ Poincaré inequality for some numbers $1 \le s \le t < +\infty$. Then there exists a constant $C=C(Q_1,Q_2,C_1,C_2,s,t)>0$ such that for any $f \in H^{1,s}(X,\dist,\meas)$,
\begin{equation}\label{global Sobolev inequality}
\left( \int_{X} |f|^t \di \mu \right)^{1/t} \le C \Ch_s(f)^{1/s}.
\end{equation}
\end{theorem}

\begin{proof}
Take $f \in H^{1,s}(X,\dist,\meas)$. Then
\begin{align}\label{eq:eq1}
\int_X |f|^t \di \mu & \le \sum_{i \in \cV} \int_{U_i} |f|^t \di \mu \nonumber \\
& \le 2^{t-1} \sum_{i \in \cV} \left( \int_{U_i} |f-f_{U_i,\mu}|^t \di \mu + \int_{U_i} |f_{U_i,\mu}|^t \di \mu  \right) \qquad \text{thanks to \eqref{eq:2p}} \nonumber \\
& \le 2^{t-1}C_1^t \sum_{i \in \cV} \left( \int_{U_i^*} |\nabla f|_{*,s}^s \di \mu \right)^{t/s} + 2^{t-1} \sum_{i \in \cV} |f_{U_i,\mu}|^t \mu(U_i) \quad \text{thanks to \eqref{continuous Sobolev-Neumann 1}.}
\end{align}
On one hand, \eqref{eq:p} and \eqref{eq:Q1} implies
\begin{align}\label{eq:eq2}
\sum_{i \in \cV} \left( \int_{U_i^*} |\nabla f|_{*,s}^s \di \mu \right)^{t/s} \le \left(  \int_X \sum_{i \in \cV} 1_{U_i^*} |\nabla f|_{*,s}^s \di \mu \right)^{t/s} \le Q_1^{t/s}\left(  \int_X |\nabla f|_{*,s}^s \di \mu \right)^{t/s} .
\end{align}
On the other hand, \eqref{discrete Poincaré} and a suitable double use of Hölder's inequality yields
\begin{align*}
\sum_{i \in \cV}  |f_{U_i,\mu}|^t \mu(U_i) & \le C_2^t \sum_{(i,j) \in \cE} ||f_{U_i,\mu}| - |f_{U_{j},\mu}||^t \mu(i,j)\\
& \le C_2^t \sum_{(i,j) \in \cE} \int_{U_i} \int_{U_j}|f(x)-f(y)|^t \di \mu(x)\di \mu(y) \frac{\mu(i,j)}{\mu(U_i)\mu(U_j)} \, \cdot
\end{align*}
As for any $(i,j)\in \cE$ there exists $k(i,j)$ such that $U_i \cup U_j \subset U_{k(i,j)}^*$ and $\mu(i,j) \le \mu(U_{k(i,j)}^*)\le Q_2\, \mu(i,j)$,
we get
\[
\sum_{i \in \cV}  |f_{U_i,\mu}|^t \mu(U_i) \le C_2^t \sum_{(i,j) \in \cE} \int_{U_{k(i,j)}^*} \int_{U_{k(i,j)}^*}|f(x)-f(y)|^t \di \mu(x)\di \mu(y) \frac{Q_2}{\mu(U_{k(i,j)}^*)}
\]
For any $(i,j) \in \cE$, applying \eqref{eq:2p} with $a=f(x) - f_{U_{k(i,j)}^*}$ and $b = f_{U_{k(i,j)}^*} - f(y)$ leads to
\[
\int_{U_{k(i,j)}^*} \int_{U_{k(i,j)}^*}|f(x)-f(y)|^t \di \mu(x)\di \mu(y) \le 
2^t \mu(U_{k(i,j)}^*) \int_{U_{k(i,j)}^*}|f-f_{U_{k(i,j)}^*}|^t \di \mu.
\]
Therefore,
\begin{align}\label{eq:eq3}
\sum_{i \in \cV}  |f_{U_i,\mu}|^t \mu(U_i) & \le (2 C_2)^t Q_2 \sum_{(i,j) \in \cE}\int_{U_{k(i,j)}^*}|f-f_{U_{k(i,j)}^*}|^t \di \mu \nonumber \\
& \le (2 C_1 C_2)^t Q_2  \sum_{(i,j) \in \cE} \left( \int_{U_{k(i,j)}^*} |\nabla f|_{*,s}^s \di \mu \right)^{t/s} \qquad \text{thanks to \eqref{continuous Sobolev-Neumann 2}} \nonumber \\
& \le (2 C_1 C_2)^t Q_2  \left( \int_X \sum_{(i,j) \in \cE} 1_{U_{k(i,j)}^*}  |\nabla f|_{*,s}^s \di \mu \right)^{t/s}  \quad \text{thanks to \eqref{eq:p}} \nonumber \\
& \le (2 C_1 C_2)^t Q_2 Q_1^{3t/s} \left( \int_X |\nabla f|_{*,s}^s \di \mu \right)^{t/s} \quad  \text{thanks to \eqref{eq:Q12}}.
\end{align}
Combining \eqref{eq:eq1}, \eqref{eq:eq2} and \eqref{eq:eq3} yields \eqref{global Sobolev inequality} with $$C=2^{t-1}[C_1^tQ_1^{t/s} + (2C_1 C_2)^t Q_2 Q_1^{3t/s}].$$

\end{proof}

\textbf{Good coverings via $\kappa$-decomposition}

In this paragraph, we provide a systematic manner to build a good covering on doubly measured metric spaces satisfying mild assumptions. 

\begin{remark}
We suggest the reader to keep in mind the set $X \subset \setR^2 \approx \mathbb{C}$ defined as
\[
B_1(0) \cup A_1 \cup A_2 \cup A_3
\]
where 
\begin{align*}
A_1 & =\{r e^{i\theta} : r>0,  \, \theta \in (-\pi/4,\pi/4)\},\\
A_2 & =\{re^{i\theta} : 0< r < 20,  \, \theta \in (\pi/2,3\pi/4)\},\\
A_3 & =  \{re^{i\theta} : 0< r < 17,  \, \theta \in (\pi,3\pi/2)\} \backslash \{re^{i\theta} : 3< r < 15,  \, \theta \in (5\pi/4,7\pi/4)\},
\end{align*}
as a guiding example to apply the construction described in this paragraph, in the case where the parameter $\kappa$ appearing thereafter is equal to $2$.
\end{remark}

Let $(X,\dist)$ be a connected metric space, $o \in X$ and $\kappa >1$.  For any fixed $i \in \setZ$,  consider the connected components $\{V_{i,a}\}_{a \in \Lambda(i)}$ of the annulus $A(o,\kappa^{i-1},\kappa^{i})$, and note that $\overline{V_{i,a}} \cap S_{\kappa^{i-1}}\neq \emptyset$ for any $a \in \Lambda(i)$.

 Write $\Gamma$ for the set of indices $(i,a)$ where $i \in \setZ$ and $a \in \Lambda(i)$.  Consider $\Lambda \subset \Gamma$ defined by 
\[
\Lambda:=\{(i,a) :  \overline{V_{i,a}} \cap S_{\kappa^{i}}(o) \neq \emptyset \}
\]
and for any $i\in \setZ$ denote by $\aleph(i)$ the set of $a \in \Lambda(i)$ such that $(i,a) \in \Lambda$. Observe that 
\[
\Gamma \backslash \Lambda = \bigcup_{\kappa^{i-1} < \eps \le \kappa^{i}} \{ (i,a)  : \overline{V_{i,a}} \cap S_\eps(o) = \emptyset \}
\]
because $\overline{V_{i,a}} \cap S_\eps(o) = \emptyset$ for some $\kappa^{i-1} < \eps \le \kappa^{i}$ implies $\overline{V_{i,a}} \cap S_{\kappa^{i}}(o) = \emptyset$. This means in particular that pieces $V_{i,a}$ where $(i,a) \in \Gamma \backslash \Lambda$ may be of arbitrary small width.  In order to build a good covering from the pieces $\{V_{i,a}\}_{(i,a) \in \Gamma}$, we want to avoid this arbitrary smallness, because it may contradict the measure control in (4) of Definition \ref{def:goodcovering}. For this reason, we proceed as follows.
\begin{itemize}
\item First we define sets $\{U_{i,a}\}_{(i,a) \in \Lambda}$ by setting $U_{i,a}:=V_{i,a}$ for any $(i,a) \in \Lambda$.
\item Secondly we enlarge some of these sets $\{U_{i,a}\}_{(i,a) \in \Lambda}$ as follows: for any $(i,a) \in \Gamma \backslash \Lambda$, we choose $b \in \aleph(i-1)$ such that $\overline{V_{i,a}} \cap \overline{V_{i-1,b}} \neq \emptyset$ and we set
$$U_{i-1,b}:= V_{i,a} \cup V_{i-1,b} \cup (\overline{V_{i,a}} \cap \overline{V_{i-1,b}}).$$
\end{itemize}
In this way, we obtain a family of sets $\{U_{i,a}\}_{(i,a) \in \Lambda}$ such that 
\begin{equation}\label{eq:prop1}
\overline{U_{i,a}} \cap S_{\kappa^{i-1}}(o) \neq \emptyset \qquad \text{and} \qquad \overline{U_{i,a}} \cap S_{\kappa^{i}}(o) \neq \emptyset
\end{equation}
for any $(i,a) \in \Lambda$. Note also that
\begin{equation}\label{eq:prop2}
X \backslash \bigcup_{(i,a) \in \Lambda} U_{i,a} \subset \bigcup_{i \in \setZ} S_{\kappa^{i}}(o).
\end{equation}

\begin{definition}
Let $(X,\dist)$ be a metric space, $o \in X$ and $\kappa >1$.  We call $\kappa$-decomposition of $X$ centered at $o$ the family of connected open sets $\{U_{i,a}\}_{(i,a) \in \Lambda}$ built as above.
\end{definition}

On geodesic doubling metric measure spaces, $\kappa$-decompositions satisfy the following key lemma.

\begin{lemma}\label{lem:1}
Let $(X,\dist,\meas)$ be such that $(X,\dist)$ is geodesic and $\meas$ is doubling of doubling dimension $Q$. Then for any $o \in X$ and $\kappa>1$,  the cardinality $h(i)$ of each layer $\{U_{i,a}\}_{a \in \aleph(i)}$ of the $\kappa$-decomposition of $X$ centered at $o$ is finite and bounded from above by a finite number $h=h(\kappa,Q)$.
\end{lemma}

\begin{proof}
Take $\kappa >1$. Take $i \in \setZ$.  Let $\{V_{i,a}\}_{a \in \Lambda(i)}$ be the connected components of the annulus $A(o,\kappa^{i-1},\kappa^{i})$.  Set $r_i:=(\kappa^{i}+\kappa^{i-1})/2$. Then
\begin{equation*}
V_{i,a} \cap S_{r_i} \neq \emptyset
\end{equation*}
for any $(i,a) \in \Lambda$.  Indeed,  this follows from the intermediate value theorem applied to the continuous function $[0,1] \ni t \mapsto \dist(o,\gamma(t))$ where $\gamma$ is a geodesic joining a point in $\overline{V_{i,a}} \cap S_{\kappa^{i-1}}(o)$ to another point in $\overline{V_{i,a}} \cap S_{\kappa^{i}}(o)$. 

Set $\rho_i:=(\kappa^{i}-\kappa^{i-1})/4$.  Note that if $x_a \in \overline{V_{i,a}} \cap S_{r_i}(o)$, then
\[
\text{$\{B_{\rho_i}(x_{a})\}_{a \in \aleph(i)}$ is a disjoint family of balls}.
\]
Indeed, since $(X,\dist)$ is geodesic these balls are all connected,  and they all entirely lie in disjoint connected components of $A(o,\kappa^{i-1},\kappa^{i})$.

Now observe that the balls $\{B_{\rho_i}(x_{a})\}_{a \in \aleph(i)}$ are all included in $B_{\kappa^{i}}(o)$. Therefore, for any integer $N>1$, if we pick $N$ such balls $\{B_{\rho_i}(x_{a_\ell})\}_{1 \le \ell \le N}$ we have
\[
N \min_{1 \le \ell \le N} \meas(B_{\rho_i}(x_{a_\ell}))  \le \sum_{1 \le \ell \le N} \meas(B_{\rho_i}(x_{a_\ell})) \le \meas(B_{\kappa^{i}}(o)).
\]
Let $b$ be an integer between $1$ and $N$ such that $\mu(B_{\rho_i}(x_{b})) = \min_{1 \le \ell \le N} \meas(B_{\rho_i}(x_{a_\ell}))$. Since $B_{\kappa^{i}}(o) \subset B_{\kappa^{i}+\dist(o,b)}(x_b)\subset B_{2\kappa^{i}}(x_b)$, we get
\[
N \le \frac{\meas(B_{\kappa^{i}}(o))}{\meas(B_{\rho_i}(x_{b}))} \le \frac{\meas(B_{2\kappa^{i}}(x_b))}{\meas(B_{\rho_i}(x_{b}))} \le 2^Q \left( \frac{2\kappa^{i}}{\rho_i}\right)^{Q} = 2^Q \left( \frac{8 \kappa}{\kappa-1} \right)^{Q}=:h.
\]
This implies that $\aleph(i)$ has a finite cardinality not greater than $h$.
\end{proof}

We are now in a position to build good coverings on a suitable class of doubly measured metric spaces.

\begin{proposition}\label{prop:11}
Let $(X,\dist,\meas,\mu)$ be such that $(X,\dist)$ is geodesic and $\meas$ is doubling with doubling dimension $Q$. Assume that $\mu$ is absolutely continuous with respect to $\meas$ with density $F(\cdot)=\meas(B_{\dist(o,\cdot)}(o))^\alpha \dist(o,\cdot)^{-\beta}$ for some $o \in X$ and $\alpha \ge 0$, $\beta>0$. Then for any $\kappa>1$,  the space $(X,\dist,\meas,\mu)$ admits a good covering with parameters $Q_1=Q_1(Q,\kappa)$ and $Q_2=Q_2(Q,\kappa,\alpha,\beta)$.
\end{proposition}

\begin{proof}

Take $\kappa>1$ and build the $\kappa$-decomposition $\{U_{i,a}\}_{(i,a) \in \Lambda}$ of $X$ centered at $o$.  For any $(i,a) \in \Lambda$, set $\Lambda^*(i,a):=\{(j,b) \in \Lambda : \overline{U_{i,a}} \cap \overline{U_{j,b}} \neq \emptyset\}$ and
\[
U_{i,a}^{*}:=\bigcup_{(j,b) \in \Lambda^*(i,a)} U_{j,b},
\]
then set $\Lambda^\#(i,a):=\{(j,b) \in \Lambda : \overline{U_{i,a}^*} \cap \overline{U_{j,b}^*} \neq \emptyset\}$ and
\[
U_{i,a}^{\#}:=\bigcup_{(j,b) \in \Lambda^\#(i,a)} U_{j,b}^*.
\]
We claim that the family of triples $\{(U_{i,a},U_{i,a}^{*},U_{i,a}^{\#})\}_{(i,a)\in \Lambda}$ is a good covering of the doubly measured metric space $(X,\dist,\meas,\mu)$. 

Indeed, for any $(i,a)\in \Lambda$, we obviously have $U_{i,a} \subset U_{i,a}^{*} \subset U_{i,a}^{\#}$. Moreover, since $\meas$ is doubling, we know from \eqref{eq:sphèresnégli} and \eqref{eq:prop2} that $\meas(X\backslash \bigcup_{i,a} U_{i,a})=0$, and since $\mu$ is absolutely continuous with respect to $\meas$ we also get $ \mu(X\backslash \bigcup_{i,a} U_{i,a}) =0$.  Hence (1) and (2) in Definition \ref{def:goodcovering} are satisfied, and we are left with proving (3) and (4).

To prove (3), take $(i,a) \in \Lambda$ and observe that $\overline{U_{i,a}^\#} \cap \overline{U_{j,b}^\#} \neq \emptyset$ implies $U_{i,a}^\# \cup U_{j,b}^\# \subset A(o,\kappa^{i-50},\kappa^{i+50})$, where $50$ is by no means optimal but enough for our purposes. Since $U_{i,a}^\#$ and  $U_{j,b}^\#$ are made of pieces $\{U_{k,c}\}$ belonging to $A(o,\kappa^{i-50},\kappa^{i+50})$,  and since the number of such pieces is bounded from above by $101 h$ where $h=h(Q,\kappa)$ is given by Lemma \ref{lem:1}, then the number of distinct pair of indices $(j,b)$ such that $\overline{U_{i,a}^\#} \cap \overline{U_{j,b}^\#} \neq \emptyset$ is bounded from above by  a number  which depends only on $h$. Thus we have (3) with $Q_1=Q_1(Q,\kappa)$.

To prove (4), take $(i,a) \in \Lambda$ and observe that the very definition of $U_{i,a}^*$ ensures that if $\overline{U_{i,a}} \cap \overline{U_{j,b}} \neq \emptyset$ then $U_{i,a} \cup U_{j,b} \subset U_{i,a}^*$, so we may choose $k((i,a),(j,b))=(i,a)$.  Moreover, for any $(j,b)$ such that $\overline{U_{i,a}} \cap \overline{U_{j,b}} \neq \emptyset$, if we let $x_{j,b}$ be a point in $\overline{U_{j,b}} \cap S_{r_j}(o)$ where $r_j=(\kappa^j+\kappa^{j-1})/2$ -- the existence of $x_{j,b}$ is guaranteed by \eqref{eq:prop1} -- then
\[
B_{\rho_j}(x_{j,b}) \subset U_{j,b}
\]
where $\rho_j=(\kappa^{j}-\kappa^{j-1})/4$, and
\[
U_{i,a}^* \subset B_{\kappa^{i+1}}(o) \subset B_{\kappa^{i+1}+\dist(o,x_{j,b})}(x_{j,b}) \subset B_{2\kappa^{i+1}}(x_{j,b}).
\]
These inclusions and the doubling condition imply
\begin{align*}
\frac{\meas(U_{i,a}^*)}{\meas(U_{j,b})}  \le \frac{\meas(B_{2\kappa^{i+1}}(x_{j,b}))}{\meas(B_{\rho_j}(x_{j,b}))} & \le 2^Q \left( \frac{8 \kappa^{i+1}}{(\kappa^j - \kappa^{j-1})}\right)^Q = 2^Q \left( \frac{8 \kappa^{i-j}}{(\kappa - 1)}\right)^Q.
\end{align*}
Since $j \in \{i-1,i+1\}$  then $i-j \le 1$ hence we get
$$
\frac{\meas(U_{i,a}^*)}{\meas(U_{j,b})} \le  2^Q \left( \frac{8 \kappa}{\kappa-1} \right)^{Q}.
$$
Moreover,  since $U_{i,a}^* \subset A(o,\kappa^{i-2},\kappa^{i+1})$, then $$F(x)\le \meas(B_{\kappa^{i+1}}(o))^\alpha \kappa^{-(i-2)\beta}$$ for any $x \in U_{i,a}^*$, and since $U_{j,b} \subset A(o,\kappa^{j-1},\kappa^{j+1})\subset A(o,\kappa^{i-2},\kappa^{i+2})$ then $$F(x)\ge \meas(B_{\kappa^{i-2}}(o))^\alpha \kappa^{-(i+2)\beta}$$ for any $x \in U_{j,b}$. Therefore,
\begin{align*}
\frac{\mu(U_{i,a}^*)}{\mu(U_{j,b})} & = \int_{U_{i,a}^*} F(x) \di \meas(x) \left(\int_{U_{j,b}} F(x) \di \meas(x) \right)^{-1}\\
& \le \meas(B_{\kappa^{i+1}}(o))^\alpha \kappa^{-(i-2)\beta}\meas(U_{i,a}^*) \bigg(\meas(B_{\kappa^{i-2}}(o))^{\alpha} \kappa^{-(i+2)\beta}\meas(U_{j,b})\bigg)^{-1}\\
&  = \left( \frac{ \meas(B_{\kappa^{i+1}}(o))}{\meas(B_{\kappa^{i-2}}(o))} \right)^{\alpha} \kappa^{4 \beta}  \frac{\meas(U_{i,a}^*)}{\meas(U_{j,b})} \\
& \le 2^{Q\alpha} \left( \frac{\kappa^{i+1}}{\kappa^{i-2}}\right)^{\alpha Q} \kappa^{4 \beta}   2^Q \left( \frac{8 \kappa}{\kappa - 1}\right)^Q \\
& \le 2^{Q(\alpha+1)} \kappa^{3\alpha Q + 4 \beta}  \left( \frac{8 \kappa}{\kappa-1} \right)^{Q}.
\end{align*}
Note that $2^{Q(\alpha+1)} \kappa^{3\alpha Q + 4 \beta}>1$. Thus we have $(4)$ with $Q_2 = 2^{Q(\alpha+1)} \kappa^{3\alpha Q + 4 \beta} \left( \frac{8 \kappa}{\kappa-1} \right)^Q$.
\end{proof}

\hfill

\textbf{RCA property}

The next definition is taken from \cite[Def.~5.1]{MR2149405}.

\begin{definition}
We say that a metric space $(X,\dist)$ has the Relatively Connected Annuli (RCA) property with respect to a point $o \in X$ if there exists $\kappa>1$ such that for any $R\ge \kappa^2$, any two points $x,y \in S_R(o)$ can be joined by a continuous path whose image is contained in $A(o,\kappa^{-1} R,\kappa R)$.
\end{definition}

Spaces satisfying the assumptions of Theorem \ref{eq:main} have the RCA property, as implied by the next proposition. 

\begin{proposition}\label{prop:RCA}
Let $(X,\dist)$ be a complete geodesic metric space equipped with a doubling measure $\meas$ such that for some $o \in X$,
\begin{enumerate}
\item there exists $p \in [1,+\infty)$, $\lambda\ge1$ and $C_P>0$ such that for all $f \in L^1_{\loc}(X,\meas)$ and $g \in \UG^{p}(f)$, for all $r>0$,
$$
\int_{B_r(o)} |f - f_{B_{r}(o)}|^p \di \meas \le C_P r^p \int_{B_{\lambda r}(o)} g^p \di \meas,
$$

\item there exists $C_o>0$ and $\eta>p$ such that for any $0<r \le R$,
$$
\frac{\meas(B_R(o))}{\meas(B_r(o))} \ge C_o \left( \frac{R}{r} \right)^\eta.
$$
\end{enumerate}
Then $(X,\dist)$ has the RCA property with respect to $o$. Moreover, the coefficient $\kappa$ of this property depends only on the doubling dimension $Q$, $p$, $\lambda$, $C_P$, $\eta$ and $C_o$.
\end{proposition}

Note that compared to the assumptions of Theorem \ref{eq:main},  in the previous Proposition we only need the weak $(p,p)$ Poincaré inequality to hold on all balls centered at $o$.

\begin{proof}
Let us fix $R\ge 1$. For the sake of clarity,  we call relative thickness of an annulus $A(o,r_1,r_2)$ the quantity $r_2 /r_1$, and for any integers $i<j$ we denote by $A(i,j)$ the annulus $A(o,2^iR, 2^jR)$. For any $l \in\setN\backslash \{0\}$, we let $i_l \in \{1,\ldots,l\}$ be:
\begin{itemize}
\item equal to $l$ if for any $i \in \{1,\ldots,l-1\}$ the annulus $A(i-1,1)$ is contained in two different connected components of the larger annulus $A(i-1,l)$,
\item equal to the greatest integer $i$ between $1$ and $l-1$ such that the annulus $A(i-1,i)$ is included in one single connected component of the larger annulus $A(i-1,l)$ in case there exists $i \in \{1,\ldots,l-1\}$ such that this property holds.
\end{itemize}
We aim at establishing
\begin{equation}\label{eq:cla}
\sup_{ l \in \setN\backslash \{0\}}( l - i_l) \le D
\end{equation}
for some $D=D(C_D,p, C_P, \eta, C_o)<+\infty$.
 Indeed,  this bound would imply that for any $l$ the maximal relative thickness of a non-connected annulus of the form $A(i,l)$ with $i \in \{1,\ldots,l\}$ is bounded from above by $2^D$.  Since $D$ is independant of $R$,  if we set $\kappa = 2^D$ we get that any annulus $A(o,\kappa^{-1}R,\kappa R)$ is connected and then path-connected since $X$ is geodesic.
 

To prove \eqref{eq:cla}, fix $l \in\setN\backslash \{0\}$ such that $l-i_l>2$.  Up to a $\meas$-negligible set, divide the ball $B_l:=B_{2^l R}(o)$ into five domains $V$, $Y_1$, $Y_2$, $Z_1$ and $Z_2$ in the following way:
\begin{itemize}
\item set $V:=B_{2^{i_l} R}(o)$,
\item by maximality of $i_l$,  the set $B_l \backslash \overline{V}=A(i_l,l)$ is not connected; choose one of its connected component and call it $W_1$; call $W_2$ the union of the other connected components;
\item finally, for any $\alpha \in \{1,2\}$, set $Y_\alpha:=W_\alpha \cap A(i,i+1)$ and $Z_\alpha:=W_\alpha \backslash A(i,i+1)$.
\end{itemize}
Apply the weak $(p,p)$ Poincaré inequality to the function $f$ defined on $B_l$ by 
$$f(x):=
\begin{cases}
0 & \text{if $x \in V$}\\
(2^{i_l}r)^{-1} \dist(o,x) - 1 & \text{if $x \in Y_1$},\\
-\meas(Z_1)\meas(Z_2)^{-1} ((2^{i_l}r)^{-1}\dist(o,x) - 1) & \text{if $x \in Y_2$},\\
1 & \text{if $x \in Z_1$},\\
-\meas(Z_1)\meas(Z_2)^{-1} & \text{if $x \in Z_2$},
\end{cases}
$$
and extended in a Lipschitz manner outside of $B_l$ (using for instance Whitney's or McShane's extension).  Note that $f$ is constructed in such a way that the mean-value $f_{Z_1 \cup Z_2}$ equals $0$.  Since $f$ is Lipschitz its local Lipschitz constant $\mathrm{lip} f $ is an upper gradient: see \cite[Prop.~1.11]{MR1708448}. Since
\[
\mathrm{lip} f (x)=
\begin{cases}
(2^{i_l}r)^{-1} & \text{if $x \in Y_1$},\\
-\meas(Z_1)\meas(Z_2)^{-1}(2^{i_l}r)^{-1} & \text{if $x \in Y_2$},\\
0 & \text{on $X \backslash (Y_1 \cup Y_2)$},
\end{cases}
\]
we get
\begin{equation}\label{eq:11}
\int_{B_l} |f - f_{B_l}|^p \di \meas \le C_P 2^{p(l-i_l)}\left( \meas(Y_1) -\frac{\meas(Z_1)^p}{\meas(Z_2)^{p}}  \meas(Y_2)\right).
\end{equation}
Convexity of the function $t \mapsto |t|^p$ applied to the quantity $|f(x) - f(y)| = |f(x) - f_{B_l} + f_{B_l} - f(y)|$ provides
\begin{align}\label{eq:12}
\int_{B_l} |f - f_{B_l}|^p \di \meas & \ge 2^{-p} \fint_{B_l} \int_{B_l} |f(x) - f(y)|^p \di \meas \nonumber\\
& \ge 2^{-p} \frac{\meas(Z_1) \meas(Z_2)}{\meas(B_l)} \left( 1 + \frac{\meas(Z_1)}{\meas(Z_2)}\right)^p.
\end{align}
Combining \eqref{eq:11} and \eqref{eq:12}, noticing that the right hand side in \eqref{eq:11} is bounded above by $C_P 2^{p(l-i_l)} \meas(Y_1\cup Y_2) (1-\meas(Z_1)^p \meas(Z_2)^{-p})$ and that $Y_1\cup Y_2 \subset B_{2^{i_l+1}}(o)$, the elementary inequality $(1-a^p)(1+a)^{-p}\le 1$ holding for any $a>0$ applied to the case $a=\meas(Z_1)/\meas(Z_2)$ eventually yields to
\begin{equation}\label{eq:13}
1 \le 2^p C_P 2^{p(l-i_l)} \frac{\meas(B_l)\meas(B_{2^{i_l+1}}(o))}{\meas(Z_1)\meas(Z_2)} \, \,.
\end{equation}
A simple reasoning based on the doubling condition shows that for any $i \in \{1,2\}$, we have $\meas(Z_i)\ge C_D^{-1} 11^{-\log_2(C_D)}\meas(B_l)$, so that
\begin{equation}\label{eq:13}
1 \le 2^p C_P 2^{p(l-i_l)} C_D^2 121^{\log_2(C_D)} \frac{\meas(B_{2^{i_l+1}}(o))}{\meas(B_l)} \, \,.
\end{equation}
Then the reverse doubling condition implies
$$
1 \le  D_o 2^{(l-i_l)(p-\eta)}
$$
where $D_o:= 4^\eta C_o C_P C_D^2 121^{\log_2 (C_D)}=4^\eta C_oC_p 484^Q$ depends only on $Q$, $C_p$, $C_o$ and $\eta$. As $\eta > p$,  \eqref{eq:cla} follows by setting $D:=(\eta - p)^{-1} \log_2(D_o)$. 
\end{proof}

\begin{remark}
Our proof yields $\kappa=[4^\eta C_o C_P 484^Q]^{(\eta-p)^{-1}}$.
\end{remark}

\textbf{A final lemma}

We conclude with an elementary lemma.

\begin{lemma}\label{lem:elem}
Let $(X,\mu)$ be a measured space, $A \subset Y$ a measurable set such that $\mu(A)>0$, and $f \in L^p(A,\mu)$ for some $p>0$. Then:
\begin{equation}\label{eq:easy}
\int_A |f-f_{A,\mu}|^p \di \mu \le 2^p \inf_{c \in \setR} \int_A |f-c|^p \di \mu.
\end{equation}
\end{lemma}

\begin{proof}
Take $c\in \setR$. Then $\|f-f_{A,\mu}\|_{L^p(A,\mu)} \le \|f-c\|_{L^p(A,\mu)} + \|c-f_{A,\mu}\|_{L^p(A,\mu)}$ and $$\|c-f_{A,\mu}\|_{L^p(A,\mu)} \le \mu(A)^{1/p}  |c- f_{A,\mu}|\le \mu(A)^{1/p-1} \int_A |c-f| \di \mu \le  \|f-c\|_{L^p(A,\mu)}$$ by Hölder's inequality. We get $\displaystyle \|f-f_{A,\mu}\|_{L^p(A,\mu)} \le 2 \inf_{c \in \setR}\|f-c\|_{L^p(A,\mu)}$ which gives \eqref{eq:easy}.
\end{proof}


\section{Local Sobolev inequalities}

For our purposes, we need suitable local Sobolev inequalities.  To prove them, we adapt arguments from \cite{HajlaszKoskela} which are based on an extension of the classical Euclidean Riesz potentials \cite{Sobolev, Stein} to the setting of metric measure spaces. 

As explained e.g.~in \cite[V.1.~\&~V.2.]{Stein}, one way to prove the local Euclidean Sobolev inequality, i.e.~the existence of a constant $C>0$ depending only on $n$ and $p \in (1,n)$ such that for any ball $B \subset \setR^n$ and any $f \in C^\infty(B)$,
\begin{equation}\label{eq:1}
\|f-f_B\|_{L^{p^\star}(B)} \le C \| |\nabla f| \|_{L^{p}(B)},
\end{equation}
 is to rely on the Riesz potential of this ball, $I_1^B$, defined by
\[
I_1^B g(x):=\int_B \frac{g(z)}{|x-z|^{n-1}} \di z
\]
for any $g \in L^1(B)$ and $x \in B$. The first step consists in establishing a so-called representation formula: for any $f \in C^\infty(B)$ and almost any $x \in B$,
\begin{equation}\label{eq:2}
|f(x)-f_B| \le C_1 I_1^B|\nabla f|(x),
\end{equation}
for some $C_1$ depending only on $n$. The second step is to apply the Hardy-Littlewood-Sobolev fractional integration theorem which states that there exists $C_2>0$ depending only on $n$ and $p$ such that
\begin{equation}\label{eq:3}
\|I_1^Bg\|_{L^{p^\star}(B)} \le C_2  \|g \|_{L^p(B)}
\end{equation}
for any $g \in L^p(B)$. Finally, raising \eqref{eq:2} to the power $p^\star$, integrating over $B$, taking the $1/p^\star$-th power of the resulting inequality and applying $\eqref{eq:3}$ gives \eqref{eq:1}. 

In \cite{HajlaszKoskela}, this strategy was successfully carried out to the context of doubling metric measure spaces thanks to a natural extension of the Riesz potential $I_1^B$.  We adapt arguments from there to prove the following.

\begin{proposition}\label{prop:localSobolev}
Let $(X,\dist,\meas)$ be a complete $\PI_p$ space. Then for any $s \in [p,Q)$, there exists a constant $C_s=C_s(Q,C_P,\lambda,s)>0$ such that for any $t \in [s,p^{\star})$,
\begin{equation}\label{eq:locSobolev}
\left( \int_B |f-f_B|^t \di \meas \right)^{1/t} \le C_s \frac{R}{\meas(B)^{1/s-1/t}} \left( \int_{B} g^s \di \meas \right)^{1/s}
\end{equation}
for any ball $B \subset X$ with radius $R>0$, any $f \in L^1_{loc}(X,\meas)$ and any $g \in UG(f)\cap L^s(B,\meas)$.
\end{proposition}

\begin{proof}
Let $s \in [p,Q)$, $t \in [s,p^{\star})$, $f \in L^1_{loc}(X,\meas)$ and $B \subset X$ with radius $R>0$ be fixed.  We begin with an overview of the proof. Our first step is to introduce an operator $$J:L^s(B,\meas) \to L^0(B,\meas)$$ which we call the $s$-Haj\l{}asz-Koskela-Riesz potential of $B$. Our second step is to establish the following representation formula: there exists $C_1=C_1(Q,C_P,\lambda)>0$ such that for any $g \in \UG(f) \cap L^s(B,\meas)$,
\begin{equation}\label{eq:step2}
|f(x)-f_B| \le C_{1}J g(x) \qquad \text{for $\meas$-a.e.~$x \in B$.}
\end{equation}
Our last step is to show the existence of $C_{2}=C_2(Q,\lambda,s)>0$ such that
\begin{equation}\label{eq:step3}
\left( \int_B |J h|^t \di \meas \right)^{1/t} \le C_{2} \frac{R}{\meas(B)^{1/s-1/t}}\left( \int_B |h|^s \di \meas \right)^{1/s}
\end{equation}
for any $h \in L^s(B,\meas)$. Raising \eqref{eq:step2} to the power $t$, integrating over $B$, taking the $(1/t)$-th power of the resulting inequality and applying $\eqref{eq:step3}$ leads to \eqref{eq:locSobolev} with $C_s=C_1C_2$. Let us now provide the details.

\hfill

\textbf{Step 1.} [Construction of the $s$-Haj\l{}asz-Koskela-Riesz potential]

Let $a$ be the center of the ball $B$. For any $x \in B\backslash \{a\}$, define a chain of balls $\{B_{i} \}_{i \in I}$ as follows. Set
\begin{equation}\label{eq:clambda}
\displaystyle c_\lambda := \frac{2\lambda - 1}{2 \lambda} \, \cdot
\end{equation}
\begin{enumerate}
\item[(A)] If $R/2 \le \dist(a,x) < R$, set $I:=\setN$ and $B_i=B_{r_i}(x_i)$ for any $i \in \setN$,  where $\{x_i,r_i\}$ are defined recursively by setting $$x_{0}:=a, \quad r_{0}:=\frac{\dist(x,a)}{2\lambda}$$ and then choosing $x_{i+1}\in \mathrm{argmin}\{\dist(y,x) : y \in S_{r_i}(x_{i})\}$ for any $i\ge 0$ and setting $$r_{i+1}:=\frac{\dist(x_{i+1},x)}{2\lambda}\, \cdot$$

\item[(B)] If $0< \dist(a,x) < R/2$, define the balls $B_{i}$ for any $i \ge 0$ as above and for any $i < 0$, take $x_{i-1}\in  \mathrm{argmax}\{\dist(y,x) : y \in S_{r_{i}}(x_{i})\}$, $$r_{i-1} := c_\lambda^{-1} r_i,$$ and $B_i:=B_{r_i}(x_i)$. Finally, set $i_o$ as the unique negative integer $i$ such that $\lambda B_{i} \subset B$ and $\lambda B_{i-1} \cap (X \backslash B) \neq \emptyset$, and $I:=\{i \in \setZ : i \ge i_o\}$.
\end{enumerate}
If $R/2 \le \dist(a,x) < R$ we set $i_o=0$.

Our construction is made in such a way that the following assertions are true.

\begin{itemize}
\item[(1)] The balls $\{\lambda B_i\}_{i \in I}$ are all contained in $B$.

\item[(2)] For any ball $B_i$ there exists a ball $W_i$ such that $W_i \subset B_i \cap B_{i+1}$ and $B_i \cup B_{i+1} \subset (2\lambda)^3 W_i$.

\item[(3)] For any $i \ge 0$, the point $x_i$ lies on a geodesic from $a$ to $x$ and
\begin{align*}
\dist(a,x) & = \dist(x,x_i)+ \dist(x_i,x_{i-1}) + \ldots + \dist(x_1,x_0)\\
& = (2 \lambda) r_i + r_{i-1} + \ldots + r_0.
\end{align*}
Then $(2 \lambda) r_{i+1} + r_{i} + \ldots + r_0 = (2 \lambda) r_i + r_{i-1} + \ldots + r_0$, hence $r_{i+1}=c_\lambda r_i.$

\item[(4)] For any $i_o \le i < 0$, the point $x_i$ lies on a geodesic from $a$ to $x_{i_o}$, and $r_{i-1}= c_\lambda^{-1} r_i$.  With the previous point (3), this implies
\begin{equation}\label{eq:ri}
r_i:= (c_\lambda)^{i} \frac{\dist(a,x)}{2\lambda} \qquad \text{for any $i \in I$}. 
\end{equation}
\item[(5)] Take $x \in B$ such that $0<\dist(a,x)<R/2$.  Then $\sum_{i=i_o}^{0} r_i + \lambda r_{i_o-1} > R$ holds and implies
\[
\sum_{i=i_o}^{0} c_\lambda^{i} + \lambda c_\lambda^{i_o-1} > R \frac{2\lambda}{\dist(a,x)}
\]
thanks to \eqref{eq:ri}. Now
\[
\sum_{i=i_o}^{0} c_\lambda^{i} + \lambda c_\lambda^{i_o-1}  = \frac{c_\lambda^{i_o}-c_\lambda}{1-c_\lambda} + \lambda c_\lambda^{i_o-1} < c_\lambda^{i_o} \left( \frac{1}{1-c_\lambda} + \frac{\lambda}{c_\lambda}\right).
\]
Therefore, if we set
\[
\omega_{\lambda}:=\frac{2\lambda}{\frac{1}{1-c_\lambda}+\frac{\lambda}{c_\lambda}} 
\]
we get
\begin{equation}\label{eq:pratique}
 \left( c_\lambda \right)^{i_o}  > \omega_{\lambda}\frac{R}{\dist(a,x)} \, \cdot
\end{equation}

\item[(6)] By Lebesgue's differentiation theorem on doubling metric measure spaces (see e.g.~\cite[Section 14.6]{HajlaszKoskela}), for $\meas$-a.e.~$x \in B$,
\begin{equation}\label{eq:x}
\lim\limits_{i \to +\infty} f_{B_i} = f(x).
\end{equation}
\end{itemize}

\hfill

\noindent For any $h \in L^s (B,\meas)$, we define the $s$-Haj\l{}asz-Koskela-Riesz potential of $h$ as:
\[
J h (x) := \sum_{i \in I} r_{i} \left( \fint_{B_{i}} |h|^s \di \meas \right)^{1/s} + R \left( \fint_{B} |h|^s \di \meas \right)^{1/s} \qquad \forall x \in B.
\]

\hfill

\textbf{Step 2.} [Proof of \eqref{eq:step2}]

Take $f \in L^1_{loc}(X,\meas)$ and $g \in \UG(f) \cap L^s(B,\meas)$. Let $x \in B\backslash \{a\}$ be such that \eqref{eq:x} is satisfied.\\

a) Assume first $R/2 \le \dist(a,x) < R$. Then
\begin{align}\label{eq:l'équation4}
|f(x) - f_{B}| & \le \sum_{i=0}^{+\infty} |f_{B_{i+1}} - f_{B_{i}}| + |f_{B_{0}}- f_B| \nonumber \\
& \le \sum_{i=0}^{+\infty}\bigg( |f_{B_{i+1}} - f_{W_{i}}| + |f_{W_{i}} - f_{B_{i}}| \bigg) + |f_{B_{0}}- f_{B}| \nonumber \\
& \le \sum_{i=0}^{+\infty} \left( \fint_{W_{i}} |f - f_{B_{i+1}}| \di \meas + \fint_{W_{i}} |f - f_{B_{i}}| \di \meas \right) +  \fint_{B_{0}} |f - f_{B}| \di \meas.
\end{align}
For any $i$, the inclusions $W_i \subset B_i$ and $B_i \subset (2\lambda)^3W_i$ and then the doubling condition imply
\begin{align*}
\fint_{W_{i}} |f - f_{B_{i}}| \di \meas & = \frac{\meas(B_i)}{\meas(W_i)} \frac{1}{\meas(B_i)} \int_{W_{i}} |f - f_{B_{i}}| \di \meas\\
&  \le \frac{\meas((2\lambda)^3W_i)}{\meas(W_i)} \frac{1}{\meas(B_i)} \int_{B_{i}} |f - f_{B_{i}}| \di \meas\\
& \le (2\lambda)^{3Q} \fint_{B_{i}} |f - f_{B_{i}}| \di \meas.
\end{align*}
By the $(p,p)$ Poincaré inequality, which implies the $(1,p)$ one with same constant $C_P$, and then Hölder's inequality, we get\footnote{this is the only place where we use $s\ge p$}
\[
\fint_{B_{i}} |f - f_{B_{i}}| \di \meas \le C_P r_i \left( \fint_{B_i} g^p \di \meas \right)^{1/p} \le C_P r_i \left( \fint_{B_i} g^s \di \meas \right)^{1/s}
\]
which leads to
\begin{equation}\label{eq:l'équation1}
\fint_{W_{i}} |f - f_{B_{i}}| \di \meas \le (2\lambda)^{3Q} C_P r_i \left( \fint_{B_i} g^s \di \meas \right)^{1/s}.
\end{equation}
In the exact same fashion we get, for any $i$,
\begin{equation}\label{eq:l'équation2}
\fint_{W_{i}} |f - f_{B_{i+1}}| \di \meas \le (2\lambda)^{3Q} C_P r_{i+1} \left( \fint_{B_{i+1}} g^s \di \meas \right)^{1/s}.
\end{equation}
Moreover, since $B \subset B_{2\dist(a,x)}(x)$, the doubling condition implies $\meas(B)/\meas(B_0) \le \meas(B_{2\dist(a,x)}(x))/\meas(B_{\dist(a,x)/2}(x)) \le C_D^2=2^{2Q}$. Combined with $B_0 \subset B$, this gives
\begin{align}\label{eq:l'équation3}
\fint_{B_{0}} |f - f_{B}| \di \meas & \le \frac{\meas(B)}{\meas(B_o)} \fint_{B} |f - f_{B}| \di \meas \nonumber \\
&  \le 2^{2Q} \fint_{B} |f - f_{B}| \di \meas \le 2^{2Q}C_P R \left( \fint_{B} g^s \di \meas \right)^{1/s}
\end{align}
where we have used again the $(1,p)$ Poincaré inequality and Hölder's inequality to get the last term. Combining \eqref{eq:l'équation1}, \eqref{eq:l'équation2} and \eqref{eq:l'équation3} with \eqref{eq:l'équation4} and observing that $2^{2Q}\le 2(2\lambda)^{3Q}$ we get
\begin{equation}\label{eq:wesh}
|f(x) - f_{B}| \le 2 (2\lambda)^{3Q}C_P Jg(x).
\end{equation}

b) Assume now $0<\dist(a,x)< R/2$. Acting as above, we get
\begin{align*}
|f(x) - f_{B}| & \le \sum_{i=i_o}^{+\infty} |f_{B_{i+1}} - f_{B_{i}}| + |f_{B_{i_o}}-f_B|\\
& \le 2(2\lambda)^{3Q}C_P \sum_{i=i_o}^{+\infty}r_i \left( \fint_{B_i} g^s \di \meas \right)^{1/s} + \fint_{B_{i_o}} |f - f_{B}| \di \meas.
\end{align*}
Our construction ensures that $B_{i_o}\subset B$ so we get
\begin{align*}
\fint_{B_{i_o}} |f - f_{B}| \le \frac{\meas(B)}{\meas(B_{i_o})} \fint_{B} |f-f_B| & \le \frac{\meas(B)}{\meas(B_{i_o})} C_P R \left( \fint_{B} g^s \di \meas \right)^{1/s}\\
& \le \frac{\meas(B_{2R}(x_{i_o}))}{\meas(B_{i_o})} C_P R \left( \fint_{B} g^s \di \meas \right)^{1/s}.
\end{align*}
In order to bound $\meas(B_{2R}(x_{i_o}))/\meas(B_{i_o})$ from above by means of the doubling condition, we look for $c>0$ such that $B_{cR}(x_{i_o}) \subset B_{i_o}$,  that is to say such that
$$
cR < \left( c_\lambda \right)^{i_o} \frac{\dist(a,x)}{2\lambda} \, \cdot
$$
It follows from \eqref{eq:pratique} that we can choose $c=\omega_{\lambda}$. Then
$$
\frac{\meas(B_{2R}(x_{i_o}))}{\meas(B_{cR}(x_{i_o}))} \le  (2/\omega_{\lambda})^{Q}.
$$
In the end we get
\begin{equation}\label{eq:wesh2}
|f(x) - f_{B}| \le \max(2(2\lambda)^{3Q},(2/\omega_{\lambda})^{Q})C_P Jg(x).
\end{equation}

To conclude, from \eqref{eq:wesh} and \eqref{eq:wesh2}, we obtain \eqref{eq:step2} with
\[
C_1 := \max(2(2\lambda)^{3Q}, (2/\omega_{\lambda})^{Q})C_P.
\]


\hfill

\textbf{Step 3.} [Proof of \eqref{eq:step3}]

Let $h \in L^s(B,\meas)$ and $x \in B\backslash \{a\}$ be fixed.  For any $\rho \in (0,\dist(x,a)/\lambda)$, let $i_\rho$ be the unique positive integer such that
\begin{equation}\label{eq:irho}
c_\lambda^{i_\rho+1} \le \frac{\lambda \rho}{\dist(a,x)} < c_\lambda^{i_\rho}
\end{equation}
where we recall that $c_\lambda\in (0,1)$ is defined in \eqref{eq:clambda}. Write
\[
Jh = J_1 h + J_2 h + R \left( \fint_{B} |h|^s \di \meas \right)^{1/s}
\]
where $\displaystyle J_1h(x) := \sum_{i=i_o}^{i_\rho} r_i \left( \fint_{B_i} |h|^s \di \meas \right)^{1/s}$ and $\displaystyle J_2h(x) := \sum_{i=i_\rho+1}^{+\infty} r_i \left( \fint_{B_i} |h|^s \di \meas \right)^{1/s}$.

\begin{claim}
There exists $C_3=C_3(Q,\lambda,s)\ge 1$ such that
\begin{equation}\label{eq:J1}
J_1h(x) \le C_3 R^{Q/s}\rho^{1-Q/s} \left( \fint_{B} |h|^s \di \meas \right)^{1/s}.
\end{equation}
\end{claim}

\begin{proof}
For any $i_o \le i \le i_\rho$ we have $B \subset B_{2R}(x_i)$ so the doubling condition implies
\[
\frac{\meas(B)}{\meas(B_{i})} \le \left( \frac{4R}{r_i}\right)^Q \, \cdot
\]
Thus
\begin{align*}
J_1h(x) & =  \sum_{i=i_o}^{i_\rho} r_i \left( \frac{\meas(B)}{\meas(B_i)} \frac{1}{\meas(B)} \int_{B_i} |h|^s \di \meas \right)^{1/s}\\
& \le (4R)^{Q/s}  \sum_{i=i_o}^{i_\rho} r_i^{1-Q/s}\left( \fint_{B} |h|^s \di \meas \right)^{1/s}.
\end{align*}
Using the concrete expression of $r_i$ given in \eqref{eq:ri} and noting that $1-Q/s<0$ implies $c_\lambda^{1-Q/s} >1$,  we get
\begin{align*}
\sum_{i=i_o}^{i_\rho} r_i^{1-Q/s} & = \left( \frac{\dist(a,x)}{2\lambda} \right)^{1-Q/s} \sum_{i=i_o}^{i_\rho} \left(c_\lambda^{1-Q/s} \right)^{i}\\
& =  \left( \frac{\dist(a,x)}{2\lambda} \right)^{1-Q/s} \frac{c_\lambda^{(1-Q/s)(i_\rho+1)}-c_\lambda^{(1-Q/s)i_o}  }{c_\lambda^{1-Q/s} -1}\\
& \le \left( \frac{\dist(a,x)}{2\lambda} \right)^{1-Q/s} \frac{c_\lambda^{(1-Q/s)(i_\rho+1)}}{c_\lambda^{1-Q/s} -1} \\
& = \left( \frac{\dist(a,x)}{2\lambda} c_\lambda^{i_\rho}\right)^{1-Q/s} \frac{1}{1 -c_\lambda^{Q/s-1}}\\
& \le  \left( \frac{\rho}{2} \right)^{1-Q/s} \frac{1}{1 -c_\lambda^{Q/s-1}} \,,
\end{align*}
where the last inequality follows from \eqref{eq:irho}.  Finally, we obtain \eqref{eq:J1} with
\[\displaystyle C_3=\frac{8^{Q/s}}{2(1-c_\lambda^{Q/s-1})} \, \cdot\]
\end{proof}

\begin{claim}\label{claim:J2}
There exists $C_4=C_4(Q,\lambda,s)\ge 1$ such that
\begin{equation}\label{eq:J2}
J_2 h (x) \le C_4 \rho \left( M|h|^s(x) \right)^{1/s}
\end{equation}
where $\displaystyle M |h|^{s}(\cdot)=\sup_{r>0} \fint_{B_r(\cdot)} |h|^s \di \meas$ is the maximal function of $|h|^{s}$.
\end{claim}

\begin{proof}
For any $i \ge i_\rho +1$,  we may use the inclusion $B_i=B_{r_i}(x_i) \subset B_{\dist(x_i,x)+r_i}(x)$ to get
\begin{align*}
\left( \fint_{B_i} |h|^s \di \meas \right)^{1/s} & \le  \left( \frac{\meas(B_{\dist(x_i,x)+r_i}(x))}{\meas(B_i)}  \fint_{B_{\dist(x_i,x)+r_i}(x)} |h|^s \di \meas \right)^{1/s}\\
& \le \left( \frac{\meas(B_{\dist(x_i,x)+r_i}(x))}{\meas(B_{r_i}(x_i))} \right)^{1/s} (M|h|^s(x))^{1/s}.
\end{align*}
Moreover,
\begin{align*}
\left( \frac{\meas(B_{\dist(x_i,x)+r_i}(x))}{\meas(B_i)} \right)^{1/s} & \le \left( \frac{\meas(B_{2\dist(x_i,x)+r_i}(x_i))}{\meas(B_{r_i}(x_i))} \right)^{1/s} \le  2^{Q/s} \left( 4\lambda+1 \right)^{Q/s}
\end{align*}
where we have used the doubling condition and the equality $\displaystyle r_i=\frac{\dist(x,x_i)}{2\lambda}\,\cdot$ Then
\[
J_2h(x) \le 2^{Q/s} \left( 4\lambda+1 \right)^{Q/s} (M|h|^s(x))^{1/s}  \sum_{i=i_\rho+1}^{+\infty} r_i .
\]
Now if we use \eqref{eq:ri} and \eqref{eq:irho}, we obtain
\begin{align*}
\sum_{i=i_\rho+1}^{+\infty} r_i  & = \frac{\dist(a,x)}{2 \lambda} \sum_{i=i_\rho+1}^{+\infty} (c_\lambda)^{i} \le \frac{\dist(a,x)}{2 \lambda} \frac{c_\lambda^{i_\rho +1}}{1-c_\lambda} \le \frac{\rho}{2(1-c_\lambda)}\, \cdot
\end{align*}
In the end, we get \eqref{eq:J2} with
\begin{equation}\label{eq:C4}
C_4=\frac{2^{Q/s} \left( 4\lambda+1 \right)^{Q/s}}{2(1-c_\lambda)} \, \cdot
\end{equation}
\end{proof}

We now combine the two previous claims to get the following.

\begin{claim}\label{claim:C6}
There exists $C_5=C_5(Q,\lambda,s)\ge 1$ such that 
\begin{equation}\label{eq:C6}
Jh(x) \le C_5 R \left(\left( \fint_{B} |h|^s \di \meas \right)^{1/s} + \left( \fint_{B} |h|^s \di \meas \right)^{1/Q} \left( M|h|^s(x) \right)^{1/s-1/Q} \right).
\end{equation}
Moreover, we can choose $C_5=2\max(C_3,C_4)$.
\end{claim}

\begin{proof}
Let $F:(0,+\infty) \to [0,+\infty]$ be defined by
\[
F(\rho):=R^{Q/s}\rho^{1-Q/s} \left( \fint_{B} |h|^s \di \meas \right)^{1/s} + \rho \left( M|h|^s(x) \right)^{1/s}
\]
for any $\rho>0$. It is easily checked that $F$ admits a global minimum
\[
F(\rho_o) =2R \left( \fint_{B} |h|^s \di \meas \right)^{1/Q}\left( M|h|^s(x) \right)^{1/s-1/Q}
\]
at
\[
\rho_o:=R  \left( \fint_{B} |h|^s \di \meas \right)^{1/Q}\left( M|h|^s(x) \right)^{-1/Q}.
\]

a) Assume \[ \rho_o < \frac{\dist(a,x)}{\lambda} \, \cdot\]
From the two previous claims, we know that for any $\rho \in (0,\dist(x,a)/\lambda)$, it holds
\begin{equation}\label{eq:N}
Jh(x) \le \max(C_3,C_4) \left(R\left( \fint_{B} |h|^s \di \meas \right)^{1/s} + F(\rho)\right)
\end{equation}
since $C_3\ge 1$ and $C_4 \ge 1$. Choosing $\rho = \rho_o$ in
 \eqref{eq:N} yields \eqref{eq:C6} with $C_5=2\max(C_3,C_4)$. 

b)  Assume\[ \rho_o \ge \frac{\dist(a,x)}{\lambda} \, \cdot\]
We may act as in the proof of Claim \ref{claim:J2} to show that
 \[
J_1h(x) + J_2h(x) \le 2^{Q/s} \left( 4\lambda+1 \right)^{Q/s} (M|h|^s(x))^{1/s}  \sum_{i=i_o}^{+\infty} r_i .
\]
The chain of balls $\{B_i\}$ is made in such a way that
 \[
  \sum_{i=i_o}^{+\infty} r_i < 2R.
 \]
Moreover, $2^{Q/s} \left( 4\lambda+1 \right)^{Q/s}\le C_4 \le \max(C_3,C_4)$. Thus
 \begin{align*}
 J_1h(x) + J_2h(x)  \le 2 \max(C_3,C_4)R (M|h|^s(x))^{1/s},
 \end{align*}
hence \eqref{eq:C6} holds $C_5=2\max(C_3,C_4)$.

\end{proof}

We are now in a position to conclude.  Claim \ref{claim:C6} together with \eqref{eq:2p} and the fact that $2^{t-1}\le 2^t$ yields
\[
\fint_B |Jh|^t \di \meas \le C_5^t 2^t R^t \left(\left( \fint_{B} |h|^s \di \meas \right)^{t/s} + \left( \fint_{B} |h|^s \di \meas \right)^{t/Q} \left( \fint_B (M|h|^s)^{\alpha} \di \meas \right) \right)
\]
with $\alpha:=t(1/s-1/Q)$. Since $\alpha >1$,  we can apply the Hardy-Littlewood maximal theorem for doubling metric measure spaces (see e.g.~\cite[Section 14.5]{HajlaszKoskela}) to $|h|^s$. This gives
\[
\fint_B (M|h|^s)^{\alpha} \di \meas \le \fint_B |h|^{s \alpha}\di \meas.
\]
Hölder's inequality with $1/\alpha = 1 + (1-\alpha)/\alpha$ implies 
\[
\fint_B |h|^{s \alpha}\di \meas \le \left( \fint_B |h|^s \di \meas \right)^{\alpha}.
\]
Therefore, observing that $t/Q +\alpha = t/s$, we get
\[
\left( \fint_{B} |h|^s \di \meas \right)^{t/Q} \fint_B (M|h|^s)^{\alpha} \di \meas \le \left( \fint_B |h|^s \di \meas \right)^{t/s}
\]
and finally
\[
\fint_B |Jh|^t \di \meas  \le C_5^t 2^t R^t \left( \fint_B |h|^s \di \meas \right)^{t/s}.
\]
Hence we get \eqref{eq:step3} with $C_2:=4\max(C_3,C_4)$, i.e.
\[
C_2 = 2^{Q/s+1}\max \left( \frac{4^{Q/s}}{1-c_\lambda^{Q/s-1}}, \frac{(4\lambda+1)^{Q/s}}{1-c_\lambda}\right)\, \cdot
\]
\end{proof}

When the space $(X,\dist,\meas)$ additionally satisfies the reverse doubling condition of Theorem \ref{eq:main}, the above constant $C_s$ can be made independent of $s$.

\begin{corollary}\label{eq:localSobolev}
Let $(X,\dist,\meas)$ be satisfying the assumptions of Theorem \ref{eq:main}. Then there exists a constant $C_{LS}= C_{LS}(Q,C_P,\lambda,\eta)>0$ such that for any $s \in [p,\eta)$, any $t \in [s,p^{\star}]$, any $f \in L^1_{loc}(X,\meas)$, any ball $B\subset X$ with radius $R>0$ and any $g \in UG(f)\cap L^s(B,\meas)$,
\begin{equation}\label{eq:claim1}
\left(\int_B|f-f_B|^t \di \meas \right)^{1/t} \le C_{LS} \frac{R}{\meas(B)^{1/s-1/t}} \left( \int_B g^s \di \meas \right)^{1/s}.
\end{equation} 
\end{corollary}

\begin{proof}
Recall from the previous proof that $C_s$ in Proposition \ref{prop:localSobolev} is equal to $C_1C_2$ where $C_1=C_1(Q,C_P,\lambda)$ and 
\[
C_2 = 2^{Q/s+1}\max \left( \frac{4^{Q/s}}{1-c_\lambda^{Q/s-1}}, \frac{(4\lambda+1)^{Q/s}}{1-c_\lambda}\right).
\]
Since $1 \le s < \eta <Q$, we have $Q/s \le Q$ and $Q/s-1 > Q/\eta-1$. Since $c_\lambda \in (0,1)$, then $c_\lambda^{Q/s-1}<c_\lambda^{Q/\eta-1}$ so we can bound $C_2$ from above by
\[
C_2':=C_2'(Q,\lambda,\eta)=2^{Q+1}\max \left( \frac{4^{Q}}{1-c_\lambda^{Q/\eta-1}}, \frac{(4\lambda+1)^{Q}}{1-c_\lambda}\right).
\]
\end{proof}

\section{Weighted Sobolev inequalities}

In this section, we prove Theorem \ref{eq:main}. Let $(X,\dist,\meas)$ be satisfying the assumptions of this theorem. By Proposition \ref{prop:RCA}, we know that $(X,\dist)$ has the RCA property with respect to $o$, and we denote by $\kappa>1$ the coefficient of this property. Take $s \in [p,\eta)$ and $t \in [s,p^{\star}]$. Recall that $\di \mu_{s,t}(\cdot)=w_{s,t}(\cdot)\di\meas(\cdot)$ where $w_{s,t}(\cdot)=\meas(B_{\dist(o,\cdot)}(o))^{t/s-1}\dist(o,\cdot)^{-t}$. As $t/s-1>p/\eta - 1>0$,  the doubly measured metric space $(X,\dist,\meas,\mu_{s,t})$ admits a good covering $\{(U_{i,a},U_{i,a}^{*},U_{i,a}^{\#})\}_{(i,a)\in \Lambda}$ constructed 
from that the $\kappa$-decomposition of $(X,\dist)$ centered at $o$ as in the proof of Proposition \ref{prop:11}. Let $(\cV,\cE,\mu)$ be the associated weighted graph in the sense of Definition \ref{def:weightedgraph}. Then the weighted Sobolev inequality \eqref{eq:weightedSobolev} is proved by the patching Theorem \ref{th:patching}, provided one shows
\begin{equation}\label{eq:goal1}
\left( \int_{U_{i,a}} |f-f_{U_{i,a},\mu_{s,t}}|^t \di \mu_{s,t} \right)^{1/t} \le C_1 \left( \int_{U_{i,a}^*} |\nabla f|_{*,s}^s \di \meas\right)^{1/s}
\end{equation}
for any $f \in L^1_{loc}(X,\meas)$ and $(i,a) \in \Lambda$, and
\begin{equation}\label{eq:goal2}
\left( \sum_{i \in \cV} |f(i)|^t \mu(i) \right)^{1/t} \le C_2 \left( \sum_{i \sim j} |f(i) - f(j)|^{t} \mu(i,j)  \right)^{1/t}
\end{equation}
for any $f:\cV \to \setR$ with finite support. Since the proof of \eqref{eq:goal1} with $U_{i,a}$ and $U_{i,a}^*$ replaced by $U_{i,a}^*$ and $U_{i,a}^\#$ respectively is similar to the one we provide for \eqref{eq:goal1}, we skip it.

\subsection{The discrete $t$ Poincaré inequality}

Let us begin with showing \eqref{eq:goal2}. To this aim, we need two known properties of the discrete $1$ Poincaré inequality on a general weighted graph, see \cite[Section 1.2]{MR2481740}. We provide proofs for the reader's convenience.  We say that a weighted graph $(\cV',\cE',\mu')$ satisfies an isoperimetric inequality if there exists $I>0$ such that
\begin{equation}\label{eq:isoperimetric inequality}
\mu'(\partial \Omega)\ge I \mu'(\Omega)
\end{equation}
for any $\Omega \subset \cV'$ with $\mu'(\Omega)<+\infty$, where $\partial \Omega := \{(i,j) \in \cE : i \in \Omega, j \notin \Omega\}$; here we use the measures $\mu'$ on $\cV$ and $\cE$ defined in Remark \ref{rk:measure}.

 \begin{lemma}\label{lemma graph theory}
Let $(\cV',\cE',\mu')$ be a weighted graph. Then:
\begin{enumerate}
\item the discrete $1$ Poincaré inequality with constant $C$ is equivalent to the isoperimetric inequality with constant $C^{-1}$;
 \item if there exists $A,B\ge 1$ such that $\sup_i \#\{j : i\sim j\} \le A$ and $\sup_{i,j} \mu'(i)/\mu'(j) \le B$, then the discrete $1$ Poincaré inequality with constant $C$ implies the $\tau$-one for any $\tau \ge 1$, with constant $2C\tau (AB)^{1-1/\tau}$.
\end{enumerate}
\end{lemma}

\begin{proof}
\textit{(1)} For any $\Omega \subset \cV'$ with $\mu'(\Omega)<+\infty$, the discrete $1$ Poincaré inequality applied to the characteristic function of $\Omega$ implies $\mu'(\Omega)\le C_2 \mu'(\partial \Omega)$, hence one side of the equivalence is proved. Conversely, if $f:\cV' \to \setR$ has a finite support, by Cavalieri's principle and the isoperimetric inequality,
\begin{align*}
\sum_{i \in \cV'} f(i) \mu'(i) & = \|f\|_{L^1(\cV',\mu')} = \int_0^{+\infty} \mu'(\{f>t\}) \di t \le I^{-1} \int_0^{+\infty} \mu'(\partial \{f>t\}) \di t\\
& = I^{-1} \int_0^{+\infty} \sum_{ i\sim j\atop f(j)\le t \le f(i)} \mu'(i,j) \di t = I^{-1} \sum_{(i, j) \in \cE'} |f(i)-f(j)| \mu'(i,j).
\end{align*}

\noindent \textit{(2)} Take $\tau \ge 1$. Let $f : \cV' \to \setR$ be with finite support. Apply the discrete $1$ Poincaré inequality to $|f|^{\tau}$:
$$
\sum_{i \in \cV'} |f(i)|^{\tau} \mu'(i) \le C \sum_{(i,j)\in \cE'} ||f(i)|^t-|f(j)|^{\tau}| \mu'(i,j).
$$
Convexity of $r \mapsto r^{\tau}$ and the mean value theorem gives $||a|^{\tau}-|b|^{\tau}| \le \tau \max(|a|,|b|)^{\tau-1} ||a|-|b||$ for any $a,b\in \setR$, and  $\max(|a|,|b|)^{\tau-1} \le (|a|^{\tau-1}+|b|^{\tau-1})$ and $||a|-|b||\le |a-b|$. From this we get
\begin{align*}
\sum_{i \in \cV'} |f(i)|^\tau \mu'(i) & \le C \tau  \sum_{(i,j)\in \cE'} (|f(i)|^{\tau-1}+|f(j)|^{\tau-1})|f(i)-f(j)| \mu'(i,j)\\
& = 2 C \tau \sum_{(i,j)\in \cE'} |f(i)|^{\tau-1}|f(i)-f(j)| \mu'(i,j)\\
& \le 2 C \tau \left( \sum_{(i,j)\in \cE'} |f(i)|^\tau \mu'(i,j)\right)^{1-1/\tau} \left(\sum_{(i,j)\in \cE'} |f(i)-f(j)|^\tau \mu'(i,j)\right)^{1/\tau},
\end{align*}
and the result follows since
\[
\sum_{(i,j)\in \cE'} |f(i)|^\tau \mu'(i,j) = \sum_{i \in \cV'} \sum_{j \sim i} |f(i)|^\tau \mu'(i,j) \le AB \sum_{i \in \cV'} |f(i)|^\tau \mu'(i).
\]
\end{proof}

We also need the following claim. Recall that $\kappa$ is the coefficient of the RCA property of $(X,\dist)$.

\begin{claim}
There exists a constant $C_e=C_e(Q,\kappa)>1$ such that for any $(i,a) \in \Lambda$,
\begin{equation}\label{Ce}
C_{e}^{-1}\frac{\meas(B_{\kappa^{i-1}}(o))^{t/s}}{\kappa^{(i+1)t}} \le \mu(i,a) \le \frac{\meas(B_{\kappa^{i+1}}(o))^{t/s} }{\kappa^{(i-1)t}}\, \cdot
\end{equation}
\end{claim}

\begin{proof}
Take $(i,a) \in \Lambda$. Then $U_{i,a} \subset A(o,\kappa^{i-1},\kappa^{i+1})$, so
\[
\meas(U_{i,a}) \frac{\meas(B_{\kappa^{i-1}}(o))^{t/s-1}}{\kappa^{(i+1)t}} \le \mu(i,a) \le \meas(U_{i,a}) \frac{\meas(B_{\kappa^{i+1}}(o))^{t/s-1}}{\kappa^{(i-1)t}} \, \cdot
\]
Bound $\meas(U_{i,a})$ from above by $\meas(B_{\kappa^{i+1}}(o))$ to get the upper bound in \eqref{Ce}. Choose $x \in U_{i,a} \cap S_{(\kappa^{i}+\kappa^{i-1})/2}(o)$ and note that $$\meas(B_{\kappa^{i-1}}(o)) \le \meas(B_{2\kappa^{i}}(x)) \le C_e \meas(B_{\rho_i}(x)) \le C_e \meas(U_{i,a}),$$ where $\rho_i:= (\kappa^{i}-\kappa^{i-1})/4$ and $C_e:= \left( 16\kappa/(\kappa - 1) \right)^{Q}$, by the doubling condition.  This leads to the lower bound in \eqref{Ce}.
\end{proof}

\begin{proposition}\label{linear isoperimetric inequality}
The isoperimetric inequality \eqref{eq:isoperimetric inequality} holds on the weighted graph $(\cV,\cE,\mu)$ with a constant $I$ depending only on $Q$, $p$, $\lambda$, $C_p$, $\eta$, $C_o$, $s$ and $t$. 
\end{proposition}

\begin{proof}
For any $i \in \setN$, let $A_i$ be the set of vertices $\{(i,a) \, : \, 1 \le a \le h(i)\} \subset \cV$. Let $\Omega \subset \cV$ be such that $\mu(\Omega)<+\infty$. Let $l$ be the greatest integer $i$ such that $A_{i} \cap \Omega \neq \emptyset$.  With no loss of generality we may assume that $l\ge 3$.  Let $a$ between $1$ and $h(l)$ be such that $(l,a) \in \Omega$. Let $e$ be an element of $\partial \Omega$ chosen in the following way:
\begin{itemize}
\item if there is an edge between $(l,a)$ and a vertex of the form $(l+1,b)$, let $e$ be this edge;
\item otherwise, as $X$ is non-compact and connected, there exists necessarily two vertices $(l+1,b')$ and $(l,b)$ that are linked by an edge. Then the RCA property of $(X,\dist)$ ensures that there exists a finite sequence of edges entirely included in $A_{l-1}\cup A_l$ that connects $(l,a)$ and $(l,b)$. Among these edges, there is necessarily one of them which belongs to $\partial \Omega$ and which we choose as $e$.
\end{itemize}

Note that in both cases, there exists an integer $b$ between $1$ and $h(l)$ such that $\mu(e) \ge \mu(l,b)$, so that \eqref{Ce} implies
\[
\mu(\partial \Omega) \ge \mu(e) \ge C_e^{-1}\frac{\meas(B_{\kappa^{l-1}}(o))^{t/s}}{\kappa^{(l+1)t}}\, \cdot
\]
Moreover, thanks to \eqref{Ce} again and the upper bound on the $h(i)$ given by Lemma \ref{lem:1},
\[
\mu(\Omega) \le \sum_{0\le i \le l} \sum_{1 \le a \le h(i)} \mu(i,a) \le C_e h \sum_{0\le i \le l} \frac{\meas(B_{\kappa^{i+1}}(o))^{t/s} }{\kappa^{(i-1)t}}\, \cdot
\]
Therefore, we can perform the next computation where the reverse doubling condition \eqref{eq:reverse} plays a crucial role:
\begin{align*}
\frac{\mu(\Omega)}{\mu(\partial \Omega)}
& \le C_{e}^2 h \sum_{0 \le i \le l} \kappa^{t(l-i+2)}\left( \frac{\meas(B_{\kappa^{i+1}}(o))}{\meas(B_{\kappa^{l-1}}(o))} \right)^{t/s} \\
& \le C_e^2  h \left(  \sum_{0 \le i \le l-3} \kappa^{t(l-i+2)}\left( \frac{\meas(B_{\kappa^{i+1}}(o))}{\meas(B_{\kappa^{l-1}}(o))}\right)^{t/s} + 1 + \sum_{l-1 \le i \le l} \kappa^{t(l-i+2)}\left( \frac{\meas(B_{\kappa^{i+1}}(o))}{\meas(B_{\kappa^{l-1}}(o))} \right)^{t/s}  \right) \\
& \le C_e^2  h \left(  \sum_{0 \le i \le l-3} C_o^{t/s} \kappa^{t(l-i+2)+(i-l) \eta t/s} + 1 + \sum_{l-1 \le i \le l}2^{Qt/s} \kappa^{t(l-i+2)+(i-l+2)Q t/s} \right) \\
& \le C_e^2  h \left( C_o^{t/s} \kappa^{2t} \sum_{0 \le i \le l-3} \kappa^{t(l-i)(1-\eta/s)} + 1 + 2^{Qt/s} \kappa^{2t}(1+\kappa^{t(1-Q/s)}) \right).
\end{align*}
Now we have
\[
 \sum_{0 \le i \le l-3} \kappa^{t(l-i)(1-\eta/s)} \le \sum_{j=0}^{+\infty} [\kappa^{t(1-\eta/s)}]^{j}=:S < +\infty
\]
since $1-\eta/s<0$. Thus
\[
\frac{\mu(\Omega)}{\mu(\partial \Omega)} \le C_e^2  h \left( C_o^{t/s} \kappa^{2t} S + 1 + 2^{Qt/s} \kappa^{2t}(1+\kappa^{t(1-Q/s)}) \right) =: I^{-1}.
\]
Now $C_e$ and $h$ depend only on $Q$ and $\kappa$, the constant $S$ depends only on $t$, $s$, $\eta$ and $\kappa$, and $\kappa$ depends only on $Q$, $p$, $\lambda$, $C_p$, $\eta$ and $C_o$. Therefore,  $I$ depends only on $Q$, $p$, $\lambda$, $C_p$, $\eta$, $C_o$, $s$ and $t$.

\end{proof}

Let us conclude.  It follows from Lemma \ref{lem:1} that there exists $A=A(Q,\kappa)\ge 1$ such that $\sup_{i \in \cV} \#\{j \in \cV : j\sim i\} \le A$. Moreover, a simple reasoning based on the doubling condition shows that there exists $B=B(Q,\kappa) \ge 1$ such that $\sup_{(i,j) \in \cE} \mu(i)/\mu(j) \le B$, see e.g.~\cite[Claim 2]{Tew}. Therefore, combining Proposition \ref{linear isoperimetric inequality} and Lemma \ref{lemma graph theory}, we get that $(\cV,\cE,\mu)$ satisfies the discrete $\tau$ Poincaré inequality for any $\tau\ge 1$, hence it satisfies \eqref{eq:goal2}, with a constant depending only on $Q$, $p$, $\lambda$, $C_p$, $\eta$, $C_o$, $s$ and $t$.

\subsection{A local Sobolev inequality on pieces of annuli}

Let us now derive a property that helps getting the continuous Sobolev-Neumann inequalities \eqref{eq:goal1}. 

\begin{proposition}\label{prop:locala}
Let $(X,\dist,\meas)$ be a complete PI space with doubling dimension $Q$ and Poincaré exponent $p<Q$. Set $p^\star:=Qp/(Q-p)$. Then for any $\alpha>1$, $\delta\in(0,1)$ and $s\in(p,p^\star)$, there exists a constant $C=C(Q,\alpha,\delta,s)>0$ such that for any $o \in X$, any $R>0$, any connected Borel subset $A$ of $A(o,R,\alpha R)$, and any $t\in(s,p^\star)$,
\[
\left( \int_A |f-f_A|^t \di \meas \right)^{1/t} \le C \frac{R}{\meas(B_R(o))^{1/s-1/t}} \left( \int_{A_\rho} g^s \di \meas \right)^{1/s}
\]
for any $f \in L^1_{loc}(X,\meas)$ and $g \in \UG(f) \cap L^s(A_\rho,\meas)$, where $\rho=\delta R$ and $A_\rho = \bigcup_{x \in A} B_\rho(x)$.
\end{proposition}

In order to prove this result, we need a suitable modification of the notion of a good covering and an associated modified patching theorem.

Let $(X,\dist,\meas)$ be a metric measure space, and let $A \subset A^{\#} \subset X$ be two Borel sets.  We call good covering of $(A,A^{\#})$ any countable family of triples $\{(U_{i},U_{i}^{*},U_{i}^{\#})\}_{i \in I}$ of Borel sets with finite $\meas$-measure satisfying (1), (3) and (4) in Definition \ref{def:goodcovering} and such that $A \backslash E \subset \bigcup_{i} U_{i} \subset \bigcup_{i} U_{i}^{\#} \subset A^{\#}$ for some Borel set $E \subset X$ with $\meas(E)=0$. For $1 \le s \le t < +\infty$, we say that such a good covering satisfies local $(t,s)$  continuous Sobolev-Neumann inequalities if there exists a constant $C_1>0$ such that for any $i$,
 \begin{equation}\label{continuous Sobolev-Neumann 2.1}
\left( \int_{U_{i}} |f-f_{U_{i}}|^t \di \meas \right)^{1/t} \le C_1 \left( \int_{U_i^*} g^s \di \meas\right)^{1/s}
\end{equation}
for any $f \in L^1_{loc}(U_i,\mu)$ and $g \in UG(f)\cap L^s(U_i^*,\meas)$, and
\begin{equation}\label{continuous Sobolev-Neumann 2.2}
\left( \int_{U_{i}^{*}} |f-f_{U_{i}^{*}}|^t \di \meas \right)^{1/t} \le C_1 \left( \int_{U_i^\#} g^s \di \meas\right)^{1/s}
\end{equation}
for any $f \in L^1_{loc}(U_i^*,\mu)$ and $g \in UG(f)\cap L^s(U_i^\#,\meas)$.

From any good covering $\{(U_{i},U_{i}^{*},U_{i}^{\#})\}_{i \in I}$ of $(A,A^{\#})$ as above, we can build a weighted graph $(\cV,\cE,\mu)$ according to Definition \ref{def:weightedgraph} and say that $\{(U_{i},U_{i}^{*},U_{i}^{\#})\}_{i \in I}$ satisfies a discrete $t$ Poincaré-Neumann inequality, where $t \ge 1$,  if there exists $C_2>0$ such that
\begin{equation}\label{eq:discretePoincaréNeumann}
\left( \sum_{i \in \cV} |f(i) - \mu(f)|^t \mu(i) \right)^{1/t} \le C_2 \left( \sum_{\{ i,j \} \in \cE} |f(i) - f(j)|^t \mu(i,j) \right)^{1/t}
\end{equation}
for all finitely supported $f:\cV \to \setR$, where $\mu(f) := \left(  \sum_{i \, : \, f(i) \neq 0} \mu(i)\right)^{-1} \sum_{i} f(i)\mu(i)$.

\begin{theorem}\label{th:patching2}
Let $(X,\dist,\meas)$ be a metric measure space, and $A\subset A^{\#} \subset X$ two Borel sets such that $0<\meas(A)<+\infty$. Assume that $(A,A^{\#})$ admits a good covering satisfying local $(t,s)$ continuous Sobolev-Neumann inequalities and a discrete $t$ Poincaré-Neumann inequality for some numbers $1 \le t \le s < +\infty$. Then there exists a constant $C>0$ such that for any $f \in L^1_{loc}(X,\meas)$ and any $g \in UG(f) \cap L^s(A^{\#},\meas)$,
\begin{equation}\label{global Sobolev inequality 2}
\left( \int_{A} |f - f_{A}|^t \di \meas \right)^{1/t} \le C \left( \int_{A^{\#}} g^s \di \meas \right)^{1/s}.
\end{equation}
Moreover, one can choose
\[
C=C_1 C'(Q_1,Q_2,C_2,s)
\]
where $C_1$ is the constant of the local continuous Sobolev-Neumann inequalities \eqref{continuous Sobolev-Neumann 2.1} and \eqref{continuous Sobolev-Neumann 2.2}.
\end{theorem}

\begin{proof}
Thanks to Lemma \ref{lem:elem}, for any $c \in \setR$,
\begin{align*}
\int_{A} |f - f_{A,\mu}|^t \di \meas & \le 2^t \int_{A} |f - c|^t \di \meas \le 2^t \sum_{i \in \cV} \int_{U_i} |f - c|^t \di \meas\\
& \le  2^t \sum_{i \in \cV} \int_{U_i} (|f - f_{U_i,\meas}| + |f_{U_i,\meas} - c|)^t \di \meas\\
& \le 2^{2t-1} \left( \sum_{i \in \cV} \int_{U_i} |f - f_{U_i,\mu}|^t\di \meas + \sum_{i \in \cV} \int_{U_i}  |f_{U_i,\meas} - c|^t \di \meas \right).
\end{align*}
Choose $c=\mu(f)$. Then both terms can be estimated as in the proof of Theorem \ref{th:patching}.
\end{proof}

\begin{remark}
A similar notion of a good covering and a similar patching theorem can be stated in the case of a doubly measured metric space $(X,\dist,\meas,\mu)$. Here we need only the case $\meas=\mu$ so we willingly omit the general case.
\end{remark}

Let $(\cV,\cE)$ be a graph. Let $\mu_c : \cV \sqcup \cE \to \setR$ be defined by $\mu_c(i)=\mu_c(e):=1$ for any $i\in \cV$ and $e \in \cE$. The associated measure on $(\cV,\cE)$, still denoted by $\mu_c$, is called counting mesure. Let us point out two simple facts. We refer to \cite[Prop.~2.12]{MR2481740} for a proof of the first one.

\begin{lemma}
For any $s\ge 1$, any finite connected graph $(\cV,\cE)$ with $N\ge 2$ vertices equipped with the counting measure satisfies a discrete $s$ Poincaré-Neumann inequality:
\begin{equation}\label{eq:lem}
\sum_{i \in \cV} |f(i)-\mu_c(f)|^s \le N(N-1)^{s-1} \sum_{(i,j)\in \cE}|f(i)-f(j)|^s
\end{equation}
for any $f:\cV \to \setR$ with finite support, where $\mu_c(f)=\left( \# \{ i : f(i) \neq 0\}\right)^{-1} \sum_{i} f(i)$.
\end{lemma}

\begin{lemma}\label{cor:discretePN}
Let $(\cV,\cE,\mu)$ be a weighted graph such that $K^{-1}L\mu_c \le \mu \le K L\mu_c$ for some $K>1$ and $L>0$. Then for any $s \ge 1$, for any $f:\cV \to \setR$ with finite support,
\[
\sum_{i \in \cV} |f(i)-\mu(f)|^s \mu(i) \le 2^sN(N-1)^{s-1} K^2 \sum_{(i,j)\in \cE}|f(i)-f(j)|^s \mu(i,j).
\]
\end{lemma}
\begin{proof}
Let $f:\cV \to \setR$ be with finite support. Using Lemma \ref{lem:elem} and \eqref{eq:lem},
\begin{align*}
\sum_{i \in \cV} |f(i)-\mu(f)|^s \mu(i) & \le K L \sum_{i \in \cV} |f(i)-\mu(f)|^s \\
& \le K L 2^s \inf_{c \in \setR}  \sum_{i \in \cV} |f(i)-c|^s\\
& \le K L 2^s \sum_{i \in \cV} |f(i)-\mu_c(f)|^s\\
& \le K L 2^s N(N-1)^{s-1} \sum_{(i,j)\in \cE}|f(i)-f(j)|^s\\
& \le K^2 2^s N(N-1)^{s-1} \sum_{(i,j)\in \cE}|f(i)-f(j)|^s \mu(i,j).
\end{align*}
\end{proof}

We are now in a position to prove Proposition \ref{prop:locala}.

\begin{proof}
Let $\{x_i\}_i$ be a $(\rho/3)$-net of $A$, that is to say a maximal set such that $A \subset \bigcup_i B_{\rho/3}(x_i)$ and all the balls $\{B_{\rho/6}(x_i)\}_i$ are disjoint one from another. Set $V_i=B_{\rho/3}(x_i)$ and $V_i^{*}=V_i^{\#}=B_\rho(x_i)$ for any $i$. Then $\{(V_i,V_i^{*},V_i^{\#})\}_i$ is a good covering of $(A,A_\rho)$ with $Q_1=60^Q$ and $Q_2=18^Q$. Indeed, $A \subset \bigcup_i V_i \subset \bigcup_i V_i^{\#} \subset A_\rho$ and $V_i \subset V_i^{*} \subset V_i^{\#}$ for any $i$. Moreover, if $c_i$ denotes the cardinality of the set of indices $j$ such that $j \sim i$, then
\[
\meas(B_{3\rho}(x_i)) \ge \meas\left(\bigcup_{j\sim i} B_\rho(x_j)\right) \ge \meas\left(\bigcup_{j\sim i} B_{\rho/6}(x_j)\right) = \sum_{j\sim i}\meas(B_{\rho/6}(x_j)) \ge c_i \min_{j\sim i} \meas(B_{\rho/6}(x_j)),
\]
thus
\[
c_i \le \frac{\meas(B_{3\rho}(x_i))}{\min_{j\sim i} \meas(B_{\rho/6}(x_j))} \le \frac{\meas(B_{5\rho}(x_{j_o}))}{\meas(B_{\rho/6}(x_{j_o}))} \le 60^Q,
\]
where $j_o$ is such that $\meas(B_{\rho/6}(x_{j_o}))=\min_{j\sim i} \meas(B_{\rho/6}(x_j))$ and where we have used the doubling condition. Finally, if $i \sim j$ we can choose $k(i,j)=i$, and then $\mu(V_i^{*}) \le 6^Q \mu(V_i)$ and
\[
\mu(V_j^*) \le \mu(B_{3\rho}(x_i)) \le 18^Q \mu(V_i)
\]
by the doubling condition. Let $(\cV,\cE,\mu)$ be the weighted graph associated with $\{(V_i,V_i^{*},V_i^{\#})\}_i$.

Since the balls $\{B_{\rho/6}(x_i)\}_i$ do not intersect and are all contained in $B_{\alpha R +\rho/6}(o)$, setting $i_o=\mathrm{argmin}_i \meas(B_{\rho/6}(x_i))$ we have
\[
|\cV| \meas(B_{\rho/6}(x_{i_o})) \le \sum_{i \in \cV} \meas(B_{\rho/6}(x_i)) \le \meas(B_{\alpha R +\rho/6}(o)) \le \meas(B_{2\alpha R +\rho/3}(x_{i_o})).
\]
Then the doubling condition implies $|\cV|\le [4(6\alpha/\delta +1)]^Q=:N(Q,\alpha,\delta)$. Moreover, it follows from the doubling condition that $\mu(V_i)/\mu(V_j) \le (1+2\alpha/\delta)^Q$ for any $i,j$. Choosing one index $i_o$ and setting $L=\mu(V_{i_o})$, we get $K^{-1}L\mu_c \le \mu \le K L\mu_c$ with $K:=(1+2\alpha/\delta)^Q$. Therefore, we can apply Lemma \ref{cor:discretePN}. We get that for any $s \ge 1$, $(\cV,\cE,\mu)$ satisfies a $s$ Poincaré-Neumann inequality with a constant $C_2=C_2(Q,\alpha,\delta,s)>0$.

Moreover, it follows from Proposition \ref{prop:localSobolev} that there exists a constant $C_1=C_1(Q,p,\lambda)>0$ such that for any $f \in L^1_{loc}(X,\meas)$,
\[
\left( \int_{V_i} |f-f_{V_i}|^t \di \meas \right)^{1/t} \le C \frac{R}{\meas(B_R(o))^{1/s-1/t}} \left( \int_{V_i^{*}} |\nabla f|_{*,s}^s \di \meas \right)^{1/s}
\]
and
\[
\left( \int_{V_i^*} |f-f_{V_i^*}|^t \di \meas \right)^{1/t} \le C \frac{R}{\meas(B_{R}(o))^{1/s-1/t}} \left( \int_{V_i^\#} |\nabla f|_{*,s}^s \di \meas \right)^{1/s}.
\]
Therefore, Theorem \ref{th:patching2} gives the desired result.
\end{proof}

\subsection{The local $(t,s)$ continuous Sobolev-Neumann inequalities}

Let us now prove \eqref{eq:goal1}. Take $(i,a) \in \Lambda$. Thanks to Lemma \ref{lem:elem},
$$
\int_{U_{i,a}} |f-f_{U_{i,a},\mu_{s,t}}|^t \di \mu_{s,t} \le 2^t \int_{U_{i,a}} |f-f_{U_{i,a},\meas}|^t \di \mu_{s,t}.
$$ 
Moreover, $x \in U_{i,a}$ implies $\kappa^{i-2}\le \dist(o,x)\le \kappa^{i+1}$ and then $w_{s,t}(x) \le \meas(B_{\kappa^{i+1}}(o))^{t/s-1}/\kappa^{(i-2)t}$, so that 
\begin{align*}
\int_{U_{i,a}} |f-f_{U_{i,a},\meas}|^t \di \mu_{s,t} & \le \frac{\meas(B_{\kappa^{i+1}}(o))^{t/s-1}}{\kappa^{(i-2)t}}\int_{U_{i,a}} |f-f_{U_{i,a},\meas}|^t \di \meas\\
& \le C(Q,\kappa,s) \frac{\meas(B_{\kappa^{i+1}}(o))^{t/s-1}}{\kappa^{(i-2)t}} \frac{\kappa^{(i-1)t}}{\meas(B_{\kappa^{i-1}}(o))^{t/s-1}} \left( \int_{U_{i,a}^{*}} g^s \di \meas \right)^{t/s}\\
& = C(Q,\kappa,s) \left(\frac{\meas(B_{\kappa^{i+1}}(o))}{\meas(B_{\kappa^{i-1}}(o))}\right)^{t/s-1}\kappa^t \left( \int_{U_{i,a}^{*}} g^s \di \meas \right)^{t/s}\\
& \le C(Q,\kappa,s) (C_D \kappa^{2Q})^{t/s-1}\kappa^t \left( \int_{U_{i,a}^{*}} g^s \di \meas \right)^{t/s},
\end{align*}
where we have applied Proposition \ref{prop:locala} with $R=\kappa^{i-1}$, $\alpha=\kappa^2$, $A = U_{i,a}$ and $\delta=\kappa/2$ to get the second inequality,  and the doubling condition to get the last one. This gives the desired local $(t,s)$ continuous Sobolev-Neumann inequality \eqref{eq:goal1} with $C_2:=C(Q,\kappa,s)^{1/t} (C_D \kappa^{2Q})^{1/s-1/t}\kappa$. As $1/s-1/t\le 1/Q$ and $C_D\ge1$, $\kappa>1$, we can take $C_2:=C(Q,\kappa,s)^{1/t} C_D^{1/Q}\kappa^3 = C(Q,\kappa,s)^{1/t} 2 \kappa^3$. Since $\kappa$ depends only on $Q$,$p$,$\lambda$,$C_P$,$\eta$ and $C_o$, then $C_2$ depends only on $Q$,$p$,$\lambda$,$C_P$,$\eta$, $C_o$, $s$ and $t$.

\section{Hardy inequalities}

In this section, we prove Theorem \ref{th:hardy}. Let $(X,\dist,\meas)$ be satisfying the assumptions of this theorem. Like in the previous section, let  $\kappa>1$ be the coefficient of the RCA property of $(X,\dist)$. Take $s \in [p,\eta)$ and set $\di \mu_{s}(\cdot):=\dist(o,\cdot)^{-s}\di\meas(\cdot)$. It follows from Proposition \ref{prop:11} that the $\kappa$-decomposition of $(X,\dist)$ centered at $o$ can be modified into a good covering $\{(U_{i,a},U_{i,a}^{*},U_{i,a}^{\#})\}_{(i,a)\in \Lambda}$ of the doubly measured metric space $(X,\dist,\meas,\mu_{s})$. Let $(\cV,\cE,\mu)$ be the associated weighted graph. Then the Hardy inequality \eqref{eq:Hardy} stems from the patching Theorem \ref{th:patching} applied with $t=s$ and $\mu=\mu_s$, provided one shows
\begin{equation}\label{eq:goal1H}
\left( \int_{U_{i,a}} |f-f_{U_{i,a},\mu_s}|^s \di \mu_s \right)^{1/s} \le C_1 \left( \int_{U_{i,a}^*} |\nabla f|_{*,s}^s \di \meas\right)^{1/s}
\end{equation}
for any $f \in L^1_{loc}(X,\meas)$ and $(i,a) \in \Lambda$, and
\begin{equation}\label{eq:goal2H}
\left( \sum_{i \in \cV} |f(i)|^s \mu(i) \right)^{1/s} \le C_2 \left( \sum_{i \sim j} |f(i) - f(j)|^{s} \mu(i,j)  \right)^{1/s}
\end{equation}
for any $f:\cV \to \setR$ with finite support.

Again, the proof of \eqref{eq:goal1H} with $U_{i,a}$ and $U_{i,a}^*$ replaced with $U_{i,a}^*$ and $U_{i,a}^\#$ respectively is similar, so we skip it. The discrete inequality \eqref{eq:goal2H} follows from the arguments in Section 4.1. The continuous inequality \eqref{eq:goal1H} can be proved in a similar way as in Section 4.3 with the help of the following local Poincaré inequality on pieces of annuli.

\begin{proposition}\label{prop:localP}
Let $(X,\dist,\meas)$ be a complete PI space with doubling dimension $Q$ and Poincaré exponent $p<Q$. Set $p^\star:=Qp/(Q-p)$. Then for any $\alpha>1$, $\delta\in(0,1)$ and $s\in(p,p^\star)$, there exists a constant $C=C(Q,\alpha,\delta,s)>0$ such that for any $o \in X$, any $R>0$ and any connected Borel subset $A$ of $A(o,R,\alpha R)$, 
\begin{equation}\label{eq:Poincaréannuli}
\int_A |f-f_A|^s \di \meas \le C R^s \int_{A_\rho} g^s \di \meas 
\end{equation}
for any $f \in L^1_{loc}(X,\meas)$ and $g \in \UG(f) \cap L^s(A_\rho,\meas)$, where $\rho=\delta R$ and $A_\rho = \bigcup_{x \in A} B_\rho(x)$.
\end{proposition}

\begin{proof}
Let $\{x_i\}_i$ be a $(\rho/\lambda)$-net of $A$. Set $V_i=B_{\rho/\lambda}(x_i)$, $V_i^{*}=B_{\rho}(x_i)$ and $V_i^{\#}=B_{\lambda \rho}(x_i)$ for any $i$. Then it is easily seen that $\{(V_i,V_i^{*},V_i^{\#})\}_i$ is a good covering of $(A,A_\rho)$ with $Q_1$ and $Q_2$ depending only on $Q$ and $\lambda$.  Thus the result follows from applying Theorem \ref{th:patching2} like in the proof of Proposition \ref{prop:locala}. With the same arguments as there, one gets that a discrete $s$ Poincaré-Neumann inequality holds for any $s\ge 1$. The required continuous Sobolev-Neumann inequalities are a direct consequence of the weak $(1,p)$ Poincaré inequality on $(X,\dist,\meas)$.
\end{proof}

We are now in a position to conclude the proof of \eqref{eq:goal1H}. Take $f \in L^1_{loc}(X,\meas)$ and $(i,a) \in \Lambda$. Then
\begin{align*}
\int_{U_{i,a}} |f-f_{U_{i,a},\mu_s}|^s \di \mu_s & \le (\kappa^{i-1})^{-s} \int_{U_{i,a}} |f-f_{U_{i,a},\mu_s}|^s \di \meas\\
& \le 2^s (\kappa^{i-1})^{-s} \inf_{c \in \setR} \int_{U_{i,a}} |f-c|^s \di \meas\\
& \le 2^s (\kappa^{i-1})^{-s} C (\kappa^{i+1})^s \int_{U_{i,a}^*} |\nabla f|_{*,s}^s \di \meas \quad \text{thanks to \eqref{eq:Poincaréannuli}}\\
& = 2^s \kappa^{2s} C \int_{U_{i,a}^*} |\nabla f|_{*,s}^s \di \meas.
\end{align*}
This concludes the proof of Theorem \ref{th:hardy}.

\bibliographystyle{alpha}
\bibliography{adimensional}

\end{document}